\documentclass[10pt]{amsart}

\numberwithin{equation}{section}
\newcounter{mnote}
\let\oldmarginpar\marginpar
\renewcommand\marginpar[1]{\-\oldmarginpar[\raggedleft\footnotesize #1]%
{\raggedright\footnotesize #1}}

\usepackage[english]{babel}
\usepackage{amssymb}
\usepackage{mathrsfs}
\usepackage{stmaryrd}
\usepackage{chemarrow}
\usepackage[norelsize]{algorithm2e}
\usepackage{enumerate}
\usepackage{graphicx}
\usepackage[all]{xy}
\usepackage{tikz}
\usetikzlibrary{arrows}
\usepackage{multirow}
\usepackage[colorinlistoftodos]{todonotes}
\usepackage[colorlinks=true, allcolors=blue]{hyperref}
\usepackage[hyperpageref]{backref} % 参考文献中backref
\usepackage{hyperref}
\hypersetup{hypertex=true,
colorlinks=true,
linkcolor=blue,
anchorcolor=blue,
citecolor=red}
\usepackage{verbatim}
\usepackage{booktabs}
\usepackage{tikz-cd}
\usepackage{subcaption}
\allowdisplaybreaks

\newtheorem{theorem}{Theorem}[section]
\newtheorem{lemma}[theorem]{Lemma}
\newtheorem{corollary}[theorem]{Corollary}

\newtheorem{example}[theorem]{Example}

\newtheorem{remark}[theorem]{Remark}

\newcommand{\normmm}[1]{{\left\vert\kern-0.25ex\left\vert\kern-0.25ex\left\vert #1
\right\vert\kern-0.25ex\right\vert\kern-0.25ex\right\vert}}
\newcommand{\dx}{\,{\rm d}x}
\newcommand{\ds}{\,{\rm d}s}
\newcommand{\dd}{\,{\rm d}}

\newcommand{\curl}{\operatorname{curl}}
\renewcommand{\div}{\operatorname{div}}

\newcommand{\tr}{\operatorname{tr}}
\newcommand{\rot}{\operatorname{rot}}

\newcommand{\sym}{\operatorname{sym}}
\newcommand{\skw}{\operatorname{skw}}

\newcommand{\Oplus}{\ensuremath{\vcenter{\hbox{\scalebox{1.5}{$\oplus$}}}}}

\begin{document}

% \title[]{FEM FOR SGE MODEL}

% \author{XUEHAI HUANG \quad ZHEQIAN TANG}
%\address{R.An,
% College of Mathematics and Physics, Wenzhou University,  325035, Wenzhou, China ({\it anrong@wzu.edu.cn}).}

\title[Nonconforming Finite Element Method for SGE]{An optimal and robust nonconforming finite element method for the strain gradient elasticity}
\author{Jianguo Huang}%
\address{School of Mathematical Sciences, and MOE-LSC, Shanghai Jiao Tong University, Shanghai 200240, China}%
\email{jghuang@sjtu.edu.cn}%
\author{Xuehai Huang}%
\address{School of Mathematics, Shanghai University of Finance and Economics, Shanghai 200433, China}%
\email{huang.xuehai@sufe.edu.cn}%
\author{Zheqian Tang}%
\address{School of Mathematics, Shanghai University of Finance and Economics, Shanghai 200433, China}%
\email{tangzq0329@163.com}%

\thanks{The second author is the corresponding author. 
The first author was partially supported by NSFC (Grant No.\ 12071289).
The second author was supported by the National Natural Science Foundation of China (Grant No.\ 12171300). %, and the Natural Science Foundation of Shanghai  (Grant No.\ 21ZR1480500).
}
%%%%% Keywords %%%%%%%%%%%
%\keywords{ }
\makeatletter
\@namedef{subjclassname@2020}{\textup{2020} Mathematics Subject Classification}
\makeatother

\begin{abstract}
An optimal and robust low-order nonconforming finite element method is developed for the strain gradient elasticity (SGE) model in arbitrary dimensions. An $H^2$-nonconforming quadratic vector-valued finite element in arbitrary dimensions is constructed,
which together with the Nitsche's technique, is applied for solving the SGE model. The resulting nonconforming finite element method is optimal and robust with respect to the Lam\'{e} coefficient $\lambda$ and the size parameter $\iota$, as confirmed by numerical results. Additionally, nonconforming finite element discretization of the smooth Stokes complex in two and three dimensions is devised.
\end{abstract}
\maketitle

%%%%%%%%%%%%%%%%%%%%%%%%%%%%%%%%%%%%%%%%
%% sect-introduction
%%%%%%%%%%%%%%%%%%%%%%%%%%%%%%%%%%%%%%%%
\section{Introduction}
In this paper we shall develop an optimal and robust nonconforming finite element method for the strain gradient elasticity (SGE) model on a bounded polytopal domain $\Omega\subset\mathbb{R}^d$ ($d\geq 2$)
% Assume $\Omega\subset\mathbb{R}^d(d\geq 2)$ is a bounded domain with Lipschitz boundary, and $\boldsymbol{f}\in L^2(\Omega;\mathbb R^d)$. The strain gradient elasticity (SGE) model is governed by
\begin{equation}\label{SGE0}
\begin{cases}
-\div ((\boldsymbol{I}-\iota^{2}\Delta)\boldsymbol{\sigma}(\boldsymbol{u}))=\boldsymbol{f} &\mbox{in} \ \Omega,\\
\boldsymbol{u}=\partial_{n}\boldsymbol{u}=\boldsymbol{0} &\mbox{on} \ \partial\Omega,
\end{cases}
\end{equation}
where $\boldsymbol{f}\in L^2(\Omega;\mathbb R^d)$ is the applied force, $\boldsymbol{u}=(u_1, \ldots, u_d)^\intercal$ is the displacement
field, $\partial_{n}\boldsymbol{u}$ is the normal derivative of $\boldsymbol{u}$, the stress $\boldsymbol{\sigma}(\boldsymbol{u})=2\mu\boldsymbol{\varepsilon}(\boldsymbol{u})+\lambda(\div\boldsymbol{u})\boldsymbol{I}$, and
$\boldsymbol{\varepsilon}(\boldsymbol{u})=(\varepsilon_{ij}(\boldsymbol{u}))_{d\times d}$ is the strain tensor field with
$\varepsilon_{ij}(\boldsymbol{u})=(\partial_j u_i+\partial_i u_j)/2$. Here $\lambda$ and $\mu$ are the Lam\'{e}
constants, $\boldsymbol{I}$ is the identity tensor field, and $\iota\in(0,1]$ denotes the size parameter of the material
under discussion.
We assume $0<\mu_0\leq\mu\leq\mu_1$ and $\lambda\geq0$, where $\mu_0$ and $\mu_1$ are fixed constants.
% \begin{comment}
% This SGE model \eqref{SGE0} proposed by Aifantis et al \cite{altan1992structure,ru1993simple} and can be seen as a simplification of the more general SGE models in \cite{mindlin1964micro},since it contains only one
% extra material parameter $\iota$ besides the Lam\'{e} constants $\lambda$ and $\mu$. This SGE model
% successfully eliminated the strain singularity of the brittle crack tip field \cite{exadaktylos1996two}, and we refer to \cite{eringen2012continuum} and \cite{hutchinson1997strain} for other strain gradient models.
% \end{comment}
The SGE model \eqref{SGE0}, proposed by Aifantis et al. \cite{altan1992structure,ru1993simple}, can be viewed as a simplification of the broader SGE models discussed in \cite{mindlin1964micro}, as it involves only one additional size parameter $\iota$, alongside the Lam\'{e} constants $\lambda$ and $\mu$. This SGE model effectively eliminates the strain singularity at the tip of brittle cracks \cite{exadaktylos1996two}. 
% For additional strain gradient models, we refer to \cite{eringen2012continuum} and \cite{hutchinson1997strain}.

The SGE model \eqref{SGE0} is a fourth-order strain gradient perturbation of linear elasticity with a small size parameter $\iota$. When $\iota$ approaches $0$, it reduces to the second-order linear elasticity problem. However, the solution of this limiting problem fails to satisfy the normal derivative boundary condition, leading to boundary layers.
%While the solution of linear elasticity fails to satisfy the normal derivative boundary condition, which leads to the phenomenon of boundary layer. 
Furthermore, as $\lambda$ tends to infinity, the elastic material approaches nearly incompressibility $\div\boldsymbol{u}=0$, which triggers the volume locking phenomenon of low-order conforming finite element methods.
Consequently, constructing $H^2$ finite elements to develop robust numerical methods for the SGE model \eqref{SGE0}, which account for both the size parameter $\iota$ and the Lam\'{e} constant $\lambda$, is highly challenging.

Several $H^2$-conforming finite element methods \cite{akarapu2006numerical, papanicolopulos2009three, torabi2018c1, zervos2001modelling, zervos2009two} and nonconforming finite element methods \cite{MR3712289, MR4296093, MR4021023, MR3912981} are designed for the SGE model~\eqref{SGE0}. 
However, there is no error analysis or the error estimates are not robust with respect to the Lam\'{e} constant $\lambda$ in these papers.
To address this issue, Ming and his collaborators \cite{MR4549866} reformulate \eqref{SGE0} as a fourth-order strain gradient perturbation of the Lam\'e system by introducing the pressure field $p=\lambda\div\boldsymbol{u}$, 
and robust nonconforming mixed finite element methods based on this idea have been developed in \cite{MR4650917,MR4549866,LiaoMing2025} for SGE models.
%based on which robust nonconforming mixed finite element methods are devised in \cite{MR4650917,MR4549866} for problem~\eqref{SGE0}.
In \cite{tianshudan}, a family of robust nonconforming finite element methods with the reduced integration technique has been designed for the primal formulation of problem~\eqref{SGE0} in two dimensions.
The robust error estimates in the energy norm in \cite{tianshudan,MR4650917,MR4549866} are only $O(h^{1/2})$, which is sharp but suboptimal.

Recently, a robust quartic nonconforming finite element method with optimal convergence was developed in \cite{ChenHuangHuang2025} for problem~\eqref{SGE0} in two dimensions. However, this approach requires both boundary conditions, 
$\boldsymbol{u}=0$ and $\partial_n\boldsymbol{u}=0$, to be imposed weakly via Nitsche's technique \cite{MR0341903,Schieweck2008,MR2917211}. The double application of Nitsche's technique complicates the discrete scheme. Furthermore, to the best of our knowledge, a robust nonconforming finite element method with optimal convergence for the primal formulation of \eqref{SGE0} in three dimensions is still unavailable.

In this paper, we aim to develop an optimal and robust low-order nonconforming finite element method tailored for the primal formulation of problem~\eqref{SGE0} in arbitrary dimensions. The key to craft these robust finite element methods lies in the construction of the finite element discretization of the short complex
\begin{equation}\label{complexdivpart}
H_0^2(\Omega;\mathbb{R}^d) \xrightarrow{\div} H_0^1(\Omega)\cap L_0^2(\Omega) \xrightarrow {}0.
\end{equation}
The conforming finite element discretization for the short complex \eqref{complexdivpart} has been meticulously constructed in two and three dimensions, as detailed in \cite{chen2022finite,MR4654617}. However, a limitation of $C^1$-conforming elements lies in their extensive degrees of freedom (DoFs) and the necessity for high-order polynomials in the shape functions.
%To this end, we construct an $H^2$-nonconforming quadratic finite element in any dimension to discretize $H^2(\Omega;\mathbb{R}^d)\cap H_0^1(\Omega;\mathbb{R}^d)$, whose shape function space $V(T)$ is a bubble function enrichment of the full quadratic polynomial space such that $Q(T):=\div V(T)=\mathbb{P}_{1}(T)\oplus{\rm span}\{b_T^{\textrm{NC}}\}$ with the nonconforming bubble function $b_T^{\textrm{NC}} = 2-(d+1)\sum^{d}_{i=0}\lambda_{i}^{2}$.
To this end, we construct an $H^2$-nonconforming quadratic finite element in any dimension. Its shape function space $V(T)$ is a bubble function enrichment of the full quadratic polynomial space such that $Q(T):=\div V(T)=\mathbb{P}_{1}(T)\oplus{\rm span}\{b_T^{\textrm{NC}}\}$, where $b_T^{\textrm{NC}} = 2-(d+1)\sum^{d}_{i=0}\lambda_{i}^{2}$ is the nonconforming bubble function. A useful feature of the nonconforming construction is that boundary conditions can be imposed flexibly through different boundary DoFs. If all boundary DoFs associated with the function values and first-order derivatives are imposed, the resulting space gives a nonconforming discretization of $H_0^2(\Omega;\mathbb R^d)$; see Remark~\ref{remark-V-h0}. In the method analyzed in this paper, only the DoFs corresponding to the function-value boundary condition are imposed, and hence the global finite element space $V_h$ is used to discretize $H^2(\Omega;\mathbb R^d)\cap H_0^1(\Omega;\mathbb R^d)$. The $H^1$-nonconforming finite element space $Q_h$ associated with $Q(T)$ is employed to discretize $H^1(\Omega)\cap L_0^2(\Omega)$.
%The $H^1$-nonconforming finite element space $Q_h$ of $Q(T)$ is employed to discretize $H^1(\Omega)\cap L_0^2(\Omega)$.
The global finite element space $V_h$ of $V(T)$ is $H^2$-nonconforming, but $H(\div)$-conforming.
The nonconforming finite element spaces $V_h$ and $Q_h$ form the discrete exact complex
\begin{equation}\label{femcomplexdivpart}
V_h \xrightarrow{\div} Q_h \xrightarrow {}0,
\end{equation}
which is a nonconforming discretization of the short complex
\begin{equation}\label{complexdivpart0}
H^2(\Omega;\mathbb{R}^d)\cap H_0^1(\Omega;\mathbb{R}^d) \xrightarrow{\div} H^1(\Omega)\cap L_0^2(\Omega) \xrightarrow {}0.
\end{equation}
%Interestingly, the complex \eqref{complexdivpart0} is not exact (cf. \cite[Theorem~3.1]{ArnoldScottVogelius1988}): 
%\begin{equation*}
%\div(H^2(\Omega;\mathbb{R}^d)\cap H_0^1(\Omega;\mathbb{R}^d)) \subsetneq H^1(\Omega)\cap L_0^2(\Omega).
%\end{equation*}
%In contrast, the discrete complex~\eqref{femcomplexdivpart} is exact.

For sufficiently smooth domains, for example when $\partial\Omega$ is of class $C^2$, the existence of a bounded right inverse of the divergence operator \cite[Theorem 3.1]{DanchinMucha2013} implies that the continuous complex \eqref{complexdivpart0} is exact. On nonsmooth domains, however, this exactness may fail. In two dimensions, the obstruction arises from corner compatibility conditions on the divergence of functions in $H^2(\Omega;\mathbb R^2)\cap H_0^1(\Omega;\mathbb R^2)$, understood in the Sobolev trace sense; see \cite{ArnoldScottVogelius1988}. A related corner difficulty for the SGE model with homogeneous simply supported boundary conditions is treated in \cite{LiaoMing2025} by a broken Hardy inequality. By contrast, since $V_h$ and $Q_h$ are nonconforming, no corner compatibility restrictions arise in two dimensions, nor do the analogous compatibility restrictions associated with codimension-two boundary faces arise in higher dimensions. This is one advantage of the nonconforming construction.

With the nonconforming finite element spaces $V_h$ and $Q_h$, we advance an optimal low-order nonconforming finite element method for problem~\eqref{SGE0}, which is robust with respect to the Lam\'{e} constant $\lambda$ and the size parameter $\iota$. Our approach employs Nitsche's technique to weakly enforce the boundary condition
$\partial_n\boldsymbol{u}=0$, while the boundary condition 
$\boldsymbol{u}=0$ is inherently incorporated into the finite element space 
$V_h$ in a strong manner.
In this sense, the combination of the discrete complex \eqref{femcomplexdivpart} with Nitsche's method preserves the relevant structure of the short complex \eqref{complexdivpart} at the discrete level.

Based on the Brezzi-Douglas-Marini (BDM) element interpolation \cite{MR4458899,ChenChenHuangWei2024,BrezziDouglasMarini1986,BrezziDouglasDuranFortin1987,nedelec1986new},
we build up interpolation operators $I_h^V: H^1_0(\Omega;\mathbb{R}^d)\rightarrow V_h$ and $I_h^Q: L_0^2(\Omega)\rightarrow Q_{h}$ satisfying the commutative property
\[
\div I_h^V\boldsymbol{v}=I_h^Q\div\boldsymbol{v} \quad \forall \ \boldsymbol{v}\in H^1_0(\Omega;\mathbb{R}^d).
\]
Utilizing this commutative property and interpolation error estimates, we derive optimal and robust error estimates for the proposed nonconforming finite element method, especially the error in the energy norm $\mathopen{\interleave}\boldsymbol{u}-\boldsymbol{u}_h\mathclose{\interleave}_{\iota,\lambda,h}=\min\{O(\iota^{-1/2}h),\allowbreak O(\iota^{1/2}+h^{2})\}$, under the stated regularity assumptions with $s=3$, regardless of the presence of boundary layers.

To illustrate the applicability of the method, we also briefly discuss the SGE model with mixed boundary conditions, while the rigorous error analysis is restricted to the fully clamped problem \eqref{SGE0}. In addition, the finite element spaces $V_h$ and $Q_h$ can be applied to nonconforming finite element discretizations of the smooth Stokes complexes in two and three dimensions.

The rest of this paper is organized as follows. In Section \ref{sec2}, we introduce some notation, bubble functions and the uniform regularity results for the SGE model. 
A new $H^2$-nonconforming finite element in any dimension and the corresponding interpolation are constructed in Section \ref{sec3}. 
In Section \ref{sec4}, we develop and analyze an optimal and robust nonconforming finite element method. 
Some numerical results are given in Section \ref{sec5} to confirm the theoretical results.
Finally, Appendix~\ref{app:complexes} provides the detailed construction of the nonconforming finite element discretizations of the smooth Stokes complexes in two and three dimensions.
%In Section \ref{sec5}, we devise nonconforming finite element complexes in two and three dimensions.
%Finally, some numerical results are given in Section \ref{sec6} to confirm the theoretical results.
% \begin{comment}
% It is worth mentioning that
% the uniform regularity of the problem \eqref{SGE0} on a convex polygon is proved in \cite{MR4650917} under
% following assumption
% \begin{assumption}
% \label{assumption}
% Assume that $\boldsymbol{f}\in H^{-1}(\Omega;\mathbb{R}^d)$. Let $\boldsymbol{u}\in H^{2}(\Omega;\mathbb{R}^d)$ be the solution of following problem:
% \begin{align*}
% \left\{
% \begin{array}{ll}
	% -\Delta(\mathcal{L}\boldsymbol{u})=\boldsymbol{f} &\rm{in} \ \Omega,\\
	% \boldsymbol{u}=\partial_{n}\boldsymbol{u}=0 &\rm{on} \ \partial\Omega,
	% \end{array}
% \right.
% \end{align*}
% where $\mathcal{L}\boldsymbol{u} = \mu\Delta\u+(\lambda+\mu)\nabla(\div\boldsymbol{u})$ is the usual Lam\'{e} operator,
% with $\lambda\in[0,\Lambda]$ and $\mu\in[\mu_0,\mu_1]$ being the Lam\'{e} constants.
% Then $\boldsymbol{u}\in H^{3}(\Omega;\mathbb{R}^d)$ and it admits the estimate
% $$\|\boldsymbol{u}\|_3\lesssim\|\boldsymbol{f}\|_{-1},$$
% where the hidden constant may depend on $\Lambda$, $\mu_0$ and $\mu_1$.
% \end{assumption}
% \end{comment}

%%%%%%%%%%%%%%%%%%%%%%%%%%%%%%%%%%%%%%%%
%% sect-Preliminaries
%%%%%%%%%%%%%%%%%%%%%%%%%%%%%%%%%%%%%%%%
\section{Preliminaries}\label{sec2}
\subsection{Notation}
Let $\Omega\subset\mathbb{R}^d$ ($d\geq 2$) be a bounded polytope with boundary $\partial\Omega$.
Given a bounded domain $D$ and an integer $m\geq 0$, denote by $H^m(D) $ the standard Sobolev space on $D$ with norm $\|\cdot\|_{m,D}$ and semi-norm $|\cdot |_{m,D}$, and $H_0^m(D)$ the closure of $C_0^\infty(D)$ with respect to $\|\cdot\|_{m,D}$. 
%The notation $(\cdot,\cdot)_D$ symbolizes the $L^2$ inner product on $D$.
The notation $(\cdot,\cdot)_D$ denotes the $L^2$ inner product on $D$. For a Banach space $V$, we denote by $V'$ its dual space. For $D= \Omega$, we abbreviate $\|\cdot \|_{m, D}$, $|\cdot |_{m, D}$ and $( \cdot , \cdot ) _D$ as $\|\cdot \|_m, |\cdot |_m$ and $(\cdot,\cdot)$, respectively. Let $\mathbb{P}_k(D)$ be the set of all polynomials on $D$ with the total degree up to the nonnegative integer $k$. In addition, set $\mathbb{B}(D;\mathbb{R}^d):=\mathbb{B}(D)\otimes\mathbb{R}^d$.
Referring to \cite{GiraultRaviart1986}, the Sobolev spaces
$H(\curl, D)$, $H_0(\curl, D)$, $H(\div, D)$, $H_0(\div, D)$ and $L^2_
0(D)$ are defined in the standard way.
Denote by $Q_{k,D}$ the standard $L^2$ projection operator from $L^2(D)$ to $\mathbb{P}_k(D)$, whose vectorial/tensorial version is also denoted by $Q_{k,D}$ if there is no confusion.

Denote by $\mathcal{T}_h=\{T\}$ a shape-regular conforming triangulation of $\Omega$ with each element being a simplex, where $h:=\max_{T\in\mathcal{T}_h}h_T$ and $h_T=\mathrm{diam}(T)$.
Let $\mathcal{F}_h$, $\mathring{\mathcal{F}}_h$, $\mathcal{E}_h$, $\mathring{\mathcal{E}}_h$, $\mathcal{V}_h$ and $\mathring{\mathcal{V}}_h$ be the set of all $(d-1)$-dimensional faces, interior $(d-1)$-dimensional faces, $(d-2)$-dimensional faces, interior $(d-2)$-dimensional faces, vertices and interior vertices, respectively.
Set $\mathcal F_h^{\partial}:=\mathcal{F}_h\backslash\mathring{\mathcal{F}}_h$, $\mathcal E_h^{\partial}:=\mathcal{E}_h\backslash\mathring{\mathcal{E}}_h$ and $\mathcal V_h^{\partial}:=\mathcal{V}_h\backslash\mathring{\mathcal{V}}_h$.
We next introduce patches for later use. 
For the $\ell$-dimensional face $f\in\Delta_{\ell}(\mathcal{T}_h)$ with $0\leq\ell\leq d-1$,
write $\omega_f$ to be the union of all simplices in $\mathcal{T}_f$, where $\mathcal{T}_f$ is the set of all simplices 
in $\mathcal{T}_h$ sharing the common $\ell$-dimensional face $f$. 
For $T\in\mathcal T_h$, set $\omega_T:=\bigcup_{\texttt{v}\in\Delta_0(T)}\omega_{\texttt{v}}$.
For a finite set $A$, denote by $\#A$ its cardinality.
For $F\in\mathring{\mathcal{F}}_h$, which is shared by two simplices $T^+$ and $T^-$ in $\mathcal{T}_F$, denote by $\boldsymbol{n}^+$ and $\boldsymbol{n}^-$ the unit outward normal to $T^+$ and $T^-$, respectively.
We preset the unit normal vector of $F$ by $\boldsymbol{n}_F=\boldsymbol{n}^+$.
Define the jump on $F$ as $[\![\boldsymbol{v}]\!]|_F:=\boldsymbol{v}|_{T^+}-\boldsymbol{v}|_{T^-}$. We also write $[\![\boldsymbol{v}]\!]|_F:=\boldsymbol{v}|_F$ for all $F\in\mathcal F_h^{\partial}$.

% Denote by $\mathcal{T}_h$ a conforming triangulation of $\Omega$ with each element being a simplex, where $h:=\max_{T\in\mathcal{T}_h}h_T$.
% Let $\mathcal{F}_h$, $\mathring{\mathcal{F}}_h$, $\mathcal{E}_h$ and $\mathring{\mathcal{E}}_h$ be the set of all $(d-1)$-dimensional faces, interior $(d-1)$-dimensional faces, $(d-2)$-dimensional faces and interior $(d-2)$-dimensional faces, respectively.
% Set $\mathcal F_h^{\partial}:=\mathcal{F}_h\backslash\mathring{\mathcal{F}}_h$ and $\mathcal E_h^{\partial}:=\mathcal{E}_h\backslash\mathring{\mathcal{E}}_h$.
% For $e\in \mathcal E_h$, denote by $\omega_e := \{T\in \mathcal T_h: e\subset T\}$ as the set of all simplices containing $e$. 
% We use $\nabla_h, \nabla_h^2$ and $(\div\div)_h$ to represent the element-wise gradient, Hessian and $\div\div$ with respect to $\mathcal T_h$.
% Consider two adjacent simplices $T_1$ and $T_2$ sharing an interior face $F$. Define the average and the jump of a function $w$ on $F$ as
% \begin{equation*}
% \{w\}:=\frac{1}{2}\big((w|_{T_1})|_F+(w|_{T_2})|_F\big),\;\; [w]:=(w|_{T_1})|_F\boldsymbol{n}_F\cdot\boldsymbol{n}_{\partial T_1}+(w|_{T_2})|_F\boldsymbol{n}_F\cdot\boldsymbol{n}_{\partial T_2}.
% \end{equation*}
% On a face $F$ lying on the boundary $\partial\Omega$, the above terms become
% \begin{equation*}
% \{w\}:=w|_F,\quad [w]:=w|_F.
% \end{equation*}

For a non-degenerate $d$-dimensional simplex $T$ and $\ell = 0,1,\ldots,d$, denote by $\Delta_{\ell}(T)$ the set of all $\ell$-dimensional subsimplices of $T$. 
Elements of $\Delta_0(T)$ are $d+1$ vertices $\texttt{v}_0, \texttt{v}_1,\ldots, \texttt{v}_d$ of $T$ and $\Delta_d(T)=\{T\}$.
%Set $\Delta_0(T):=\Delta_0(T)$ as the set of vertices.
% Similarly, for $f\in\mathcal{F}^r(T)$,
% define
% $$\Delta_{d-2}(F):=\{e\in\mathcal{F}^{r+1}(T):e\subset\partial f\}.$$
For each $f\in\Delta_{\ell}(T)$ with $0\leq \ell\leq d$, choose an orthonormal basis $\boldsymbol{t}_1^f,\ldots,\boldsymbol{t}_{\ell}^f$ for the tangent space of $f$, and an orthonormal basis $\boldsymbol{n}_1^f,\ldots,\boldsymbol{n}_{d-\ell}^f$ for the normal space of $f$.
We abbreviate $\boldsymbol{t}^e_1$ as $\boldsymbol{t}_e$ or $\boldsymbol{t}$ when $\ell = 1$, and $\boldsymbol{n}^F_1$ as $\boldsymbol{n}_F$ or $\boldsymbol{n}$ when $\ell = d-1$. We also abbreviate $\boldsymbol{n}^f_i$
and $\boldsymbol{t}^f_i$ as $\boldsymbol{n}_i$ and $\boldsymbol{t}_i$, respectively, if not causing any confusion.
For each face $F$, denote its tangent space by $\mathscr T^F:=\operatorname{span}\{\boldsymbol t_1^F,\ldots,\boldsymbol t_{d-1}^F\}$.
For a surface $F$ and $e\in\partial F$, denote by $\boldsymbol{n}_{F,e}$ the unit vector tangent to $F$ and outward normal to $\partial F$ along $e$.
For $i,j=0, 1, \ldots, d$, set $\boldsymbol{t}_{i,j}:=\texttt{v}_j-\texttt{v}_i$.

Given a face $F\in\Delta_{d-1}(T)$ and a vector $\boldsymbol{v}\in\mathbb{R}^d$, define 
\begin{align*}
\Pi_F\boldsymbol{v} = (\boldsymbol{I}-\boldsymbol{n}_F\boldsymbol{n}^\intercal_F)\boldsymbol{v}
\end{align*}
as the projection of $\boldsymbol{v}$ onto the face $F$. Then
\begin{align*}
% \label{divF}
\div\boldsymbol{v} = \div((\boldsymbol{n}_F\boldsymbol{n}^\intercal_F)\boldsymbol{v}+\Pi_F\boldsymbol{v}) = \partial_{\boldsymbol{n}_F}(\boldsymbol{v}\cdot\boldsymbol{n}_F) + \div_F\boldsymbol{v},
\end{align*}
where the surface divergence is denoted by $\div_F\boldsymbol{v} := \div(\Pi_F\boldsymbol{v})$.
Denote by $\nabla_F:=\Pi_F\nabla$ the surface gradient on $F$.
For a $d\times d$ matrix $A$, we denote the skew-symmetric part of $A$ as
\[\skw A:=\dfrac{1}{2}(A-A^\intercal).\]
Denote by $\mathbb{K}$ the space of all skew-symmetric $d\times d$ matrices.

We use $\nabla_h$ and $\boldsymbol{\varepsilon}_h$ to represent the elementwise version of $\nabla$ and $\boldsymbol{\varepsilon}$ with respect to $\mathcal T_h$.
For a piecewise smooth vector-valued function $\boldsymbol{v} \in H(\div,\Omega)$, define the broken stress as follows:
\[
\boldsymbol{\sigma}_h(\boldsymbol{v})=2\mu\boldsymbol{\varepsilon}_h(\boldsymbol{v})+\lambda(\div\boldsymbol{v})\boldsymbol{I}.
\]
For $s\geq 1$, define $H^s(\mathcal{T}_h;\mathbb{R}^d):=H^s(\mathcal{T}_h)\otimes\mathbb{R}^d$, where
\[
H^s(\mathcal{T}_h):=\{v\in L^2(\Omega): v|_T\in H^s(T) \quad \forall \ T\in\mathcal{T}_h\}.
\]
For a piecewise smooth scalar, vector-valued or tensor-valued function $v$, define the broken squared norm and seminorms, for $s\geq 1$, by
\[
\|v\|_{s,h}^2:=\sum_{T\in\mathcal{T}_h}\|v\|^2_{s,T}, \;\; |v|_{s,h}^2:=\sum_{T\in\mathcal{T}_h}|v|^2_{s,T},\;\;
\mathopen{\interleave} v\mathclose{\interleave}^2_{1,h}:=|v|^2_{1,h}+\sum_{F\in\mathcal F_h^{\partial}}h^{-1}_F\|v\|^2_{0,F}.
\]
For $\boldsymbol{v}\in H^s(\mathcal{T}_h;\mathbb{R}^d)\cap H(\div,\Omega)$ with $s\geq2$, 
we introduce the additional discrete norms
\begin{align*}
\|\boldsymbol{v}\|^2_{\iota,\lambda,h}&:=2\mu\|\boldsymbol{\varepsilon}_h(\boldsymbol{v})\|^2_{0}+\lambda\|\div\boldsymbol{v}\|^2_{0}+\iota^2(2\mu|\boldsymbol{\varepsilon}_h(\boldsymbol{v})|^2_{1,h}+\lambda|\div\boldsymbol{v}|^2_{1,h}),\\
\mathopen{\interleave}\boldsymbol{v}\mathclose{\interleave}^2_{\iota,\lambda,h}&:=2\mu\|\boldsymbol{\varepsilon}_h(\boldsymbol{v})\|^2_{0}+\lambda\|\div\boldsymbol{v}\|^2_{0}+\iota^2(2\mu\mathopen{\interleave}\boldsymbol{\varepsilon}_h(\boldsymbol{v})\mathclose{\interleave}^2_{1,h} + \lambda\mathopen{\interleave}\div\boldsymbol{v}\mathclose{\interleave}^2_{1,h}).
\end{align*}
For $\boldsymbol v\in H^m(\Omega;\mathbb R^d)$ with $m\geq2$, we use the shorthand $\|\boldsymbol v\|_{\iota,\lambda}:=\|\boldsymbol v\|_{\iota,\lambda,h}$.
Furthermore, we define the broken $H^1$ inner product
\begin{equation*}
(u, v)_{1,h}:=(\nabla_hu,\nabla_hv)-\sum_{F\in\mathcal F_h^{\partial}}(u,\partial_{n}v)_F-\sum_{F\in\mathcal F_h^{\partial}}(\partial_{n}u,v)_F + \eta\sum_{F\in\mathcal F_h^{\partial}}h^{-1}_F(u, v)_F.
\end{equation*}
The penalty constant $\eta$ will be chosen to be sufficiently large to ensure the coercivity of the broken $H^1$ inner product $(\cdot, \cdot)_{1,h}$ on piecewise polynomial spaces.

In this paper, we use $\lesssim$ to represent $\leq C$, where $C$ is a generic positive constant
independent of the mesh size $h$, the size parameter $\iota$ and the Lam\'{e} constant $\lambda$.
And $a\eqsim b$ means $a\lesssim b\lesssim a$.

\begin{lemma}
For $s\geq 1$, it holds
\begin{equation}\label{eq:brokenH1seminormequiv}
|v|^2_{1,h}+\sum_{F\in\mathcal F_h}h^{-1}_F\|[\![v]\!]\|^2_{0,F}\eqsim |v|^2_{1,h}+\sum_{F\in\mathcal F_h}h^{-1}_F\|Q_{0,F}[\![v]\!]\|^2_{0,F}\quad\forall~v\in H^s(\mathcal{T}_h).
\end{equation}
\end{lemma}
\begin{proof}
It is evident that the expression on the right-hand side in \eqref{eq:brokenH1seminormequiv} is not greater than that on the left-hand side.
On the other hand, by the estimate of $Q_{0,F}$,
\begin{align*}
\sum_{F\in\mathcal F_h}h^{-1}_F\|[\![v]\!]\|^2_{0,F}&\lesssim \sum_{F\in\mathcal F_h}h^{-1}_F\|[\![v]\!]-Q_{0,F}[\![v]\!]\|^2_{0,F}+\sum_{F\in\mathcal F_h}h^{-1}_F\|Q_{0,F}[\![v]\!]\|^2_{0,F} \\
&\lesssim |v|^2_{1,h}+\sum_{F\in\mathcal F_h}h^{-1}_F\|Q_{0,F}[\![v]\!]\|^2_{0,F},
\end{align*}
which indicates \eqref{eq:brokenH1seminormequiv}.
\end{proof}

Recall some notations in two dimensions. For a scalar function $v$ and vector function $\boldsymbol{v} = (v_1, v_2)^\intercal$, denote
\[
\curl v = \left(\frac{\partial v}{\partial y}, -\frac{\partial v}{\partial x}\right)^{\intercal}, \quad 
\rot\boldsymbol{v} = \frac{\partial v_2}{\partial x}-\frac{\partial v_1}{\partial y}.\]

\subsection{Bubble functions}
Let $F_i\in\Delta_{d-1}(T)$ be the $(d-1)$-dimensional face opposite to vertex $\texttt{v}_i$, and $\lambda_i$ be the barycentric coordinate corresponding to vertex $\texttt{v}_i$. Then $\lambda_i(\boldsymbol{x})$ is a linear polynomial and $\lambda_i|_{F_i}=0$.
Let $e_{ij}:=F_i\cap F_j\in\Delta_{d-2}(T)$.
Let $b_T=\lambda_0 \lambda_1\cdots \lambda_d=\lambda_i b_{F_i} = \lambda_i \lambda_j b_{e_{ij}}$ be the bubble function of $T$,
where $b_{F_i}$ is the bubble function of $F_i$ and $b_{e_{ij}}$ is the bubble function of $e_{ij}$. %Clearly, $Q_{0,T}b_T=\frac{d!}{(2d+1)!}$.
For $e=e_{ij}\in\Delta_{d-2}(T)$, define
\[
\boldsymbol m_e:=\lambda_i\boldsymbol n_{F_j,e}+\lambda_j\boldsymbol n_{F_i,e}.
\]

We generalize the $H^1$-nonconforming quadratic bubble function in two and three dimensions in \cite{FortinSoulie1983,Fortin1985} to arbitrary dimensions, i.e., let $b_T^{\textrm{NC}} = 2-(d+1)\sum^{d}_{i=0}\lambda_{i}^{2}$ for a $d$-dimensional simplex $T$.

\begin{lemma}
We have %$Q_{0,T}b_T^{\textrm{NC}}=\frac{2}{d+2}$ and
\begin{equation}\label{bTNCpropdiv}
\div\bigg(\dfrac{1}{d+2}\Big(b_T^{\rm NC}-2\lambda_{0}+\dfrac{2}{d}\Big)\sum^d_{i=1}\lambda_i\boldsymbol{t}_{0,i}\bigg)=b_{T}^{\rm NC}.
\end{equation}
\end{lemma}
\begin{proof}
Let $q=b_T^{\textrm{NC}}-2\lambda_{0}+\dfrac{2}{d}$ for simplicity.
By $\boldsymbol{t}_{0,i}\cdot\nabla \lambda_j=\delta_{i,j}-\delta_{0,j}$ with $\delta_{i,j}$ being the Kronecker delta (cf. \cite[(2.4)]{MR4458899}),
\begin{equation*}
\boldsymbol{t}_{0,i}\cdot\nabla q = \boldsymbol{t}_{0,i}\cdot\nabla b_T^{\textrm{NC}} + 2 = 2(d+1)(\lambda_{0}-\lambda_{i})+2,\quad \forall\,i=1,\ldots, d.
\end{equation*}
It follows from $\lambda_{0}+\lambda_{1}+\dots+\lambda_{d}=1$ that
\begin{align*}
\div\bigg(q\sum^d_{i=1}\lambda_i\boldsymbol{t}_{0,i}\bigg) &= q\sum^d_{i=1}\boldsymbol{t}_{0,i}\cdot\nabla \lambda_i + \sum^d_{i=1}\lambda_i\boldsymbol{t}_{0,i}\cdot\nabla q \\
&= dq + 2d\lambda_0 + 2b_T^{\textrm{NC}} - 2 = (d+2)b_T^{\textrm{NC}}.
\end{align*}
Thus, \eqref{bTNCpropdiv} holds.
\end{proof}

\begin{lemma}
It holds
\begin{equation}\label{bTNCprop}
(b_{T}^{\rm NC}, q)_F=0\quad \forall~q\in\mathbb P_1(F), F\in\Delta_{d-1}(T).
\end{equation}
\end{lemma}
\begin{proof}
Without loss of generality, assume $F$ is opposite to vertex $\texttt{v}_0$, so it's equivalent to prove
\begin{equation*}
(b_{T}^{\rm NC}, \lambda_i)_F=0,\quad i=1,2,\ldots, d.
\end{equation*}
By a direct computation,
\begin{equation*}
\frac{1}{|F|}(b_{T}^{\rm NC}, \lambda_i)_F=\frac{2}{d}-(d+1)\sum_{j=1}^d\frac{1}{|F|}\int_F\lambda_j^2\lambda_i\dd s=\frac{2}{d}-(d+1)\frac{(d-1)!}{(d+2)!}(2d+4)=0,
\end{equation*}
as required.
\end{proof}

\subsection{Weak formulation and regularity}
The weak formulation of problem \eqref{SGE0} is to find $\boldsymbol{u}\in H^2_0(\Omega;\mathbb{R}^d)$ such that
\begin{align}\label{weak1}
\iota^2a(\boldsymbol{u},\boldsymbol{v})+b(\boldsymbol{u},\boldsymbol{v})=(\boldsymbol{f},\boldsymbol{v}) \quad \forall \ \boldsymbol{v}\in H^2_0(\Omega;\mathbb{R}^d),
\end{align}
where
\begin{equation*}
a(\boldsymbol{u},\boldsymbol{v}):=(\nabla \boldsymbol{\sigma}(\boldsymbol{u}), \nabla \boldsymbol{\varepsilon}(\boldsymbol{v})), \quad
b(\boldsymbol{u},\boldsymbol{v}):=(\boldsymbol{\sigma}(\boldsymbol{u}), \boldsymbol{\varepsilon}(\boldsymbol{v})).
\end{equation*}

Taking $\iota=0$,  problem \eqref{SGE0} becomes the linear elasticity problem
\begin{equation}\label{SGElinear}
\begin{cases}
-2\mu\div(\boldsymbol{\varepsilon}(\boldsymbol{u}_0))-\lambda \nabla\div\boldsymbol{u}_0=\boldsymbol{f} &\mbox{in} \ \Omega,\\
\boldsymbol{u}_0=\boldsymbol{0} &\mbox{on} \ \partial\Omega.
\end{cases}
\end{equation}
We assume the linear elasticity problem \eqref{SGElinear} has the $s$-regularity with $2\leq s\leq 3$
\begin{equation}\label{elasregularity}
\|\boldsymbol{u}_0\|_s+\lambda\|\div\boldsymbol{u}_0\|_{s-1}\lesssim\|\boldsymbol f\|_{s-2}.
\end{equation}
When $\Omega$ is convex in two and three dimensions, the regularity \eqref{elasregularity} for $s=2$ can be found in \cite{brenner1992linear, MR2641539}.

Throughout this paper, we assume that the solution of problem~\eqref{SGE0} satisfies the regularity estimates
\begin{align}
\label{Regularity-u}
|\boldsymbol{u}-\boldsymbol{u}_0|_1+\iota\|\boldsymbol{u}\|_2+\iota^2\|\boldsymbol{u}\|_3&\lesssim\iota^{1/2}\|\boldsymbol{f}\|_0,\\
\label{Regularity-divu}
\lambda\|\div(\boldsymbol{u}-\boldsymbol{u}_0)\|_0+\lambda\iota|\div\boldsymbol{u}|_1+\lambda\iota^2\|\div\boldsymbol{u}\|_2&\lesssim\iota^{1/2}\|\boldsymbol{f}\|_0.
\end{align}
The estimates~\eqref{Regularity-u}--\eqref{Regularity-divu} were established for polytopal domains in~\cite{MR4650917} under reasonable assumptions, and for smooth domains in~\cite{MR4549866}.

\subsection{Mixed boundary conditions}
\label{subsec:mixed-bc}
We briefly discuss mixed boundary conditions for the SGE model. Since $\Omega$ is a bounded polytope, its boundary consists of finitely many relatively open boundary faces. We decompose
\[
\partial\Omega=\overline{\Gamma_{\rm cl}\cup\Gamma_{\rm ss}\cup\Gamma_{\rm fr}},
\]
where $\Gamma_{\rm cl}$, $\Gamma_{\rm ss}$, and $\Gamma_{\rm fr}$ denote the clamped, simply supported, and free parts of the boundary, respectively. We assume that these three parts are pairwise disjoint relatively open unions of boundary faces and that
$\operatorname{meas}(\Gamma_{\rm cl}\cup\Gamma_{\rm ss})>0$. 
Let $\mathcal G_{\rm fr}$ denote the collection of relatively open $(d-1)$-dimensional boundary faces contained in $\Gamma_{\rm fr}$. The SGE model with mixed boundary conditions reads
\begin{equation*}
\begin{cases}
-\div\bigl((\boldsymbol I-\iota^2\Delta)\boldsymbol\sigma(\boldsymbol u)\bigr)
=
\boldsymbol f
& \mbox{in } \Omega,\\[1mm]
\boldsymbol u=\boldsymbol g_{\rm cl}^0,\quad
\partial_n\boldsymbol u=\boldsymbol g_{\rm cl}^1
& \mbox{on } \Gamma_{\rm cl},\\[1mm]
\boldsymbol u=\boldsymbol g_{\rm ss}^0,\quad
\iota^2\partial_n\boldsymbol\sigma(\boldsymbol u)\boldsymbol n
=
\boldsymbol g_{\rm ss}^1
& \mbox{on } \Gamma_{\rm ss},\\[1mm]
\bigl(\boldsymbol\sigma(\boldsymbol u)-\iota^2\Delta\boldsymbol\sigma(\boldsymbol u)\bigr)\boldsymbol n
-
\iota^2\div_G\bigl(\partial_{n}\boldsymbol\sigma(\boldsymbol u)\bigr)
=
\boldsymbol g_{\rm fr}^0
& \mbox{on } G\in\mathcal G_{\rm fr},\\[1mm]
\iota^2\partial_n\boldsymbol\sigma(\boldsymbol u)\boldsymbol n
=
\boldsymbol g_{\rm fr}^1
& \mbox{on } \Gamma_{\rm fr},\\[1mm]
\sum_{G\in\mathcal G_{\rm fr}, e\subset\partial G}
\iota^2(\partial_{n_G}\boldsymbol\sigma(\boldsymbol u))\boldsymbol n_{G,e}
=
\boldsymbol g_{\rm fr}^{\rm sk}
& \mbox{on } e\in\mathcal E_{\rm fr}.
\end{cases}
\end{equation*}
Here $\boldsymbol n$ denotes the unit outward normal to $\partial\Omega$. For each $G\in\mathcal G_{\rm fr}$, $\div_G$ denotes the row-wise surface divergence on $G$. 
% For $e\subset\partial G$, $\boldsymbol n_{G,e}$ denotes the outward unit vector tangent to $G$ and normal to $\partial G$ along $e$. 
The set $\mathcal E_{\rm fr}$ denotes the collection of codimension-two boundary faces contained in the intersections of the closures of two adjacent free boundary faces whose normal traces are distinct along their intersection. We assume that $\boldsymbol f\in L^2(\Omega;\mathbb R^d)$ and the essential boundary data are compatible in the sense that the following affine space is nonempty:
\[
V_g:=
\bigl\{
\boldsymbol v\in H^2(\Omega;\mathbb R^d):
\boldsymbol v=\boldsymbol g_{\rm cl}^0,\ 
\partial_n\boldsymbol v=\boldsymbol g_{\rm cl}^1
\text{ on }\Gamma_{\rm cl},\quad
\boldsymbol v=\boldsymbol g_{\rm ss}^0
\text{ on }\Gamma_{\rm ss}
\bigr\}.
\]
Set
\[
V_0:=
\bigl\{
\boldsymbol v\in H^2(\Omega;\mathbb R^d):
\boldsymbol v=\boldsymbol0 \text{ on } \Gamma_{\rm cl}\cup\Gamma_{\rm ss},\quad
\partial_n\boldsymbol v=\boldsymbol0 \text{ on } \Gamma_{\rm cl}
\bigr\}.
\]
Let
$\Gamma_*:=\Gamma_{\rm ss}\cup\Gamma_{\rm fr}$, and define
$\boldsymbol g_m$ by
$
\boldsymbol g_m=\boldsymbol g_{\rm ss}^1$ on $\Gamma_{\rm ss}$ and $\boldsymbol g_m=\boldsymbol g_{\rm fr}^1$ on $\Gamma_{\rm fr}$.
For any $\boldsymbol v\in V_0$, define
\[
\tr_{\rm mix}(\boldsymbol v)
:=
\Bigl(
\boldsymbol v|_{\Gamma_{\rm fr}},
\partial_n\boldsymbol v|_{\Gamma_*},
\{\boldsymbol v|_e\}_{e\in\mathcal E_{\rm fr}}
\Bigr).
\]
We define the trace space by
\[
V_0^{\tr}:=\tr_{\rm mix}(V_0),
\]
and equip $V_0^{\tr}$ with the quotient norm
\[
\|\boldsymbol\phi\|_{V_0^{\tr}}
:=
\inf\bigl\{
\|\boldsymbol w\|_2:
\boldsymbol w\in V_0,\ 
\tr_{\rm mix}(\boldsymbol w)=\boldsymbol\phi
\bigr\}.
\]
The first component of $\tr_{\rm mix}(\boldsymbol v)$ is the trace of $\boldsymbol v$ on the free boundary $\Gamma_{\rm fr}$, the second component is the trace of $\partial_n\boldsymbol v$ on $\Gamma_*$, and the third component is the codimension-two trace of $\boldsymbol v$ on the free-boundary skeleton $\mathcal E_{\rm fr}$. Thus $V_0^{\tr}$ is the trace space defined as the image of $V_0$ under $\tr_{\rm mix}$, and its components must be compatible traces of the same function in $V_0$; it is not an unconstrained product of the component trace spaces. This image-space viewpoint is consistent with the trace-space framework for fourth-order problems in \cite{FuhrerHeuerNiemi2019}.

%The natural boundary data are collected as
%\[
%\boldsymbol g:=(\boldsymbol g_{\rm fr}^0,\boldsymbol g_m,\boldsymbol g_{\rm fr}^{\rm sk})\in (V_0^{\rm tr})',
%\]
%where $\boldsymbol g_{\rm fr}^0$ is the generalized-traction datum on $\Gamma_{\rm fr}$, $\boldsymbol g_m$ on $\Gamma_*$ is the double-traction datum, and $\boldsymbol g_{\rm fr}^{\rm sk}$ is the natural datum on the free-boundary skeleton. For any $\boldsymbol v\in V_0$, write
%\[
%\langle \boldsymbol g,\operatorname{tr}_{\rm mix}\boldsymbol v\rangle_{\rm mix}
%:=
%\langle\boldsymbol g_{\rm fr}^0,\boldsymbol v\rangle_{\Gamma_{\rm fr}}
%+
%\langle\boldsymbol g_m,\partial_n\boldsymbol v\rangle_{\Gamma_*}
%+
%\sum_{e\in\mathcal E_{\rm fr}}
%\langle\boldsymbol g_{\rm fr}^{\rm sk},\boldsymbol v\rangle_e .
%\]
The natural boundary data are understood as a linear functional on $V_0^{\tr}$. More precisely, the data $(\boldsymbol g_{\rm fr}^0,\boldsymbol g_m,\boldsymbol g_{\rm fr}^{\rm sk})$ are assumed to define a functional $\boldsymbol g\in (V_0^{\tr})'$ by
\[
\langle \boldsymbol g,\tr_{\rm mix}\boldsymbol v\rangle_{\rm mix}
:=
\langle\boldsymbol g_{\rm fr}^0,\boldsymbol v\rangle_{\Gamma_{\rm fr}}
+
\langle\boldsymbol g_m,\partial_n\boldsymbol v\rangle_{\Gamma_*}
+
\sum_{e\in\mathcal E_{\rm fr}}
\langle\boldsymbol g_{\rm fr}^{\rm sk},\boldsymbol v\rangle_e ,
\quad \boldsymbol v\in V_0 .
\]
Here $\langle\cdot,\cdot\rangle$ denotes the corresponding duality pairing, $\boldsymbol g_{\rm fr}^0$ is the generalized-traction datum on $\Gamma_{\rm fr}$, $\boldsymbol g_m$ is the double-traction datum on $\Gamma_*$, and $\boldsymbol g_{\rm fr}^{\rm sk}$ is the natural datum on the free-boundary skeleton.
For any $\boldsymbol g\in (V_0^{\tr})'$, define
\[
\|\boldsymbol g\|_{\iota,\partial}
:=
\sup_{\boldsymbol v\in V_0\setminus\{\boldsymbol0\}}
\frac{
\left|
\langle \boldsymbol g,\tr_{\rm mix}\boldsymbol v\rangle_{\rm mix}
\right|
}{
\|\boldsymbol v\|_1+\iota|\boldsymbol v|_2
}.
\]
% We assume that $\|\boldsymbol g\|_{\iota,\partial}\lesssim 1$ for $\iota\in(0,1]$, which ensures that the natural boundary functional is bounded with respect to the energy norm as $\iota\to0$.

Then the weak formulation is to find $\boldsymbol u\in V_g$ such that
\begin{align}\label{weak-mix-nonhomogeneous}
\iota^2a(\boldsymbol u,\boldsymbol v)+b(\boldsymbol u,\boldsymbol v)
&=(\boldsymbol f,\boldsymbol v)
+\langle \boldsymbol g,\tr_{\rm mix}\boldsymbol v\rangle_{\rm mix},
\quad \forall\,\boldsymbol v\in V_0 .
\end{align}
Since $\operatorname{meas}(\Gamma_{\rm cl}\cup\Gamma_{\rm ss})>0$, Korn's inequality holds on $V_0$. Together with the identity
%(cf. \cite[(7)]{MR4296093}) 
\begin{equation}\label{eq:20260614}
|\nabla\boldsymbol\varepsilon(\boldsymbol v)|^2=\frac14|\nabla^2\boldsymbol v|^2+\frac34|\sym(\nabla^2\boldsymbol v)|^2,
\end{equation}
where $\sym$ denotes the full symmetrization of third-order tensors, this gives
\[
\|\boldsymbol v\|_1+\iota|\boldsymbol v|_2
\lesssim
\|\boldsymbol v\|_{\iota,\lambda},
\quad \forall\,\boldsymbol v\in V_0.
\]
%The generalized-traction term is already controlled by the $H^1$ part of this estimate. Indeed, by the trace theorem,
%\[
%|\langle\boldsymbol g_{\rm fr}^0,\boldsymbol v\rangle_{\Gamma_{\rm fr}}|
%\lesssim
%\|\boldsymbol g_{\rm fr}^0\|_{(H^{1/2}(\Gamma_{\rm fr};\mathbb R^d))'}
%\|\boldsymbol v\|_1 .
%\]
%The double-traction terms and the free-boundary skeleton term are the
%higher-order natural terms for which the $\iota|\boldsymbol v|_2$ part is needed.
By the definition of $\|\boldsymbol g\|_{\iota,\partial}$,
\[
\left|
\langle \boldsymbol g,\tr_{\rm mix}\boldsymbol v\rangle_{\rm mix}
\right|
\lesssim
\|\boldsymbol g\|_{\iota,\partial}
\|\boldsymbol v\|_{\iota,\lambda}.
\]
Thus the natural boundary functional in \eqref{weak-mix-nonhomogeneous} is bounded with respect to the energy norm. Taking any lifting $\boldsymbol u_g\in V_g$ and writing $\boldsymbol u=\boldsymbol u_g+\boldsymbol w$ with $\boldsymbol w\in V_0$, the problem is reduced to a variational problem on $V_0$. The bilinear form $\iota^2a(\cdot,\cdot)+b(\cdot,\cdot)$ is bounded and coercive on $V_0$ with respect to $\|\cdot\|_{\iota,\lambda}$, and hence \eqref{weak-mix-nonhomogeneous} is well-posed by the Lax-Milgram lemma \cite{Ciarlet1978}.

%%%%%%%%%%%%%%%%%%%%%%%%%%%%%%%%%%%%%%%%
%% FEM AND INTERPOLATION
%%%%%%%%%%%%%%%%%%%%%%%%%%%%%%%%%%%%%%%%
\section{Nonconforming finite elements and interpolation operators}\label{sec3}
In this section, we will construct low-order nonconforming finite elements for the vector field
$\boldsymbol{u}$ and its divergence $\div\boldsymbol{u}$, along with their corresponding interpolation operators. We will also establish the error estimates for these interpolations.

\subsection{Nonconforming finite elements}\label{NCFEM}

For a $d$-dimensional simplex $T$,
take
\begin{align*}%\label{vshape}
\begin{split}
V(T)&:=\mathbb{P}_{2}(T;\mathbb{R}^{d})
\oplus \textrm{span}\bigg\{\dfrac{1}{d+2}\Big(b_T^{\textrm{NC}}-2\lambda_{0}+\dfrac{2}{d}\Big)\sum^d_{i=1}\lambda_i\boldsymbol{t}_{0,i}\bigg\} \\
&\quad\,\oplus\Oplus_{i=0}^{d}\Oplus_{j=1}^{d-1}\div\big(\skw(\boldsymbol{n}_{F_i}\otimes\boldsymbol{t}_j^{F_i})b_{T}b_{F_i}\mathbb{P}_{1}(F_i)\big)\\
&\quad\,\oplus\Oplus_{i=0}^{d}\Oplus_{j=1}^{d-1}\div\big(\skw(\boldsymbol{n}_{F_i}\otimes\boldsymbol{t}_j^{F_i})b^2_{T}b_{F_i}\mathbb{P}_{0}(F_i)\big)\\
&\quad\,\oplus\Oplus_{e\in\Delta_{d-2}(T)}\Oplus_{k=1}^{d-2}\div\big(\skw(\boldsymbol m_e\otimes\boldsymbol t_k^e)b_e^2\mathbb P_0(e)\big).
\end{split}
\end{align*}
as the space of shape functions.
Clearly, $V(T)\subset \mathbb{P}_{3d+1}(T;\mathbb{R}^{d})$. 

The tensor potentials in the last three summands of $V(T)$
are skew-symmetric and have vanishing normal traces on $\partial T$.
For the face bubbles, this follows from $b_T|_{\partial T}=0$.
For the bubbles associated with $e$, both $\boldsymbol m_e$ and
$\boldsymbol t_k^e$ are tangential to each of the two faces containing $e$,
whereas $b_e$ vanishes on the remaining faces.
For any smooth skew-symmetric tensor field $\boldsymbol\tau$,
we have $\div\div\boldsymbol\tau=0$ and
\[
(\div\boldsymbol\tau)\cdot\boldsymbol n_F
=-\div_F(\boldsymbol\tau\boldsymbol n_F).
\]
Hence, their divergences are divergence-free and have
vanishing normal traces on $\partial T$.
Moreover, these tensor potentials vanish to at least third
order at every vertex of $T$, so the corresponding bubble
functions and their first-order derivatives vanish at all vertices.
By \eqref{bTNCpropdiv}, 
\begin{equation}\label{eq:divVT}
\div V(T)=\mathbb{P}_{1}(T)\oplus{\rm span}\{b_T^{\textrm{NC}}\}.
\end{equation}
For $\boldsymbol{v}\in V(T)$ and $F\in\Delta_{d-1}(T)$, we have $\boldsymbol{v}\cdot\boldsymbol{n}|_F\in \mathbb P_2(F)$.
The directness of the sum defining $V(T)$ follows from
the coefficient-elimination argument in the proof of Lemma~\ref{dof-uni-solvent}.
Consequently,
\begin{align*}
\dim V(T)
&=d\left(\begin{matrix}d+2\\2\end{matrix}\right)+1+(d-1)d(d+1)+(d-1)(d+1)
+(d-2)\left(\begin{matrix}d+1\\2\end{matrix}\right)\\
&=d(2d^2+2d-1).
\end{align*}

\begin{remark}\rm
\label{rem:bubble-curl}
In two dimensions,
\begin{align*}
\div\big(\skw(\boldsymbol{n}_{F_i}\otimes\boldsymbol{t}_j^{F_i})b_{T}b_{F_i}\mathbb{P}_{1}(F_i)\big)&=\curl(b_{T}b_{F_i}\mathbb{P}_{1}(F_i)), \\
\div\big(\skw(\boldsymbol{n}_{F_i}\otimes\boldsymbol{t}_j^{F_i})b^2_{T}b_{F_i}\mathbb{P}_{0}(F_i)\big)&=\curl(b^2_{T}b_{F_i}\mathbb{P}_{0}(F_i)).
\end{align*}
In three dimensions,
\begin{align*}
\div\big(\skw(\boldsymbol{n}_{F_i}\otimes\boldsymbol{t}_j^{F_i})b_{T}b_{F_i}\mathbb{P}_{1}(F_i)\big)&=\curl(b_{T}b_{F_i}\mathbb{P}_{1}(F_i)\boldsymbol{t}_{3-j}^{F_i}), \\
\div\big(\skw(\boldsymbol{n}_{F_i}\otimes\boldsymbol{t}_j^{F_i})b^2_{T}b_{F_i}\mathbb{P}_{0}(F_i)\big)&=\curl(b^2_{T}b_{F_i}\mathbb{P}_{0}(F_i)\boldsymbol{t}_{3-j}^{F_i}),
\end{align*}
and
\[
\div\big(\skw(\boldsymbol m_e\otimes\boldsymbol t_e)b_e^2\mathbb P_0(e)\big)
=\curl\big(b_e^2\mathbb P_0(e)(\boldsymbol m_e\times\boldsymbol t_e)\big).
\]
In the above identities, the symbol ``$=$'' is understood as an equality of the resulting polynomial spaces, rather than a pointwise identity for each individual function.
\end{remark}

The degrees of freedom (DoFs) are chosen as
\begin{subequations}\label{dof}
\begin{align}
\label{dof1}
\int_{e}\boldsymbol{v} \dd s, & \quad e\in\Delta_{d-2}(T),\\
\label{dof2}
\int_{F}\boldsymbol{v}\cdot \boldsymbol{n} \, q\, \dd s, & \quad q\in \mathbb{P}_0(F)\oplus\hat{\mathbb{P}}_2(F), F\in\Delta_{d-1}(T),\\
\label{dof3}
\int_{F}\boldsymbol{v}\cdot\boldsymbol{t}_i \, q\, \dd s, & \quad q\in\mathbb{P}_1(F), F\in\Delta_{d-1}(T), i = 1,2,\ldots,d-1,\\
\label{dof4}
\int_{F}\div\boldsymbol{v}\,\dd s, & \quad F\in\Delta_{d-1}(T),\\
\label{dof5}
\int_{F}\partial_{n}(\boldsymbol{v}\cdot\boldsymbol{t}_i)\,\dd s, & \quad F\in\Delta_{d-1}(T), i = 1,2,\ldots,d-1,\\
\label{dof6}
\dfrac 1{d+1}\sum_{i=0}^{d}(\skw\nabla\boldsymbol{v})(\texttt{v}_{i}),&
\end{align}
\end{subequations}
where $\hat{\mathbb{P}}_2(F) = \{q\in\mathbb{P}_2(F)\cap L_0^2(F): \int_e q \ds = 0, \forall \ e\in\Delta_{d-2}(F)\}$. 
In two dimensions, DoF \eqref{dof1} is exactly $\boldsymbol{v}(\texttt{v})$ for each vertex $\texttt{v}\in\Delta_0(T)$, and $\hat{\mathbb{P}}_2(F)=\{0\}$.
In three dimensions, for face $F=F_3$, we have
\begin{equation*}
\hat{\mathbb{P}}_2(F)={\rm span}\{3(\lambda_1^2-\lambda_0^2)-2(\lambda_1-\lambda_0), 3(\lambda_2^2-\lambda_0^2)-2(\lambda_2-\lambda_0)\}.
\end{equation*}

\begin{lemma}\label{p2f-lemma}
For face $F\in\Delta_{d-1}(T)$, the DoFs
\begin{subequations}\label{p2f-dof}
\begin{align}
\label{p2f-dof1}
\int_e v \ds, & \quad e\in\Delta_{d-2}(F),\\
\label{p2f-dof2}
\int_F v \ds, & \\
\label{p2f-dof3}
\int_F v\, q\ds, & \quad q\in\hat{\mathbb{P}}_2(F)
\end{align}
\end{subequations}
are unisolvent for $\mathbb{P}_2(F)$.
\end{lemma}
\begin{proof}
It suffices to prove that the DoFs (\ref{p2f-dof1}) and (\ref{p2f-dof2}) are linearly
independent in the dual space of $\mathbb{P}_2(F)$.
Without loss of generality, we assume the vertices of $F$ are $\texttt{v}_1, \ldots, \texttt{v}_d$.
Assume that there exist constants $C_0,\ldots,C_{d}$ satisfying the following equation
\[
\sum^{d}_{i=1}C_i\int_{e_i} v \ds+C_{0}\int_F v \ds = 0 \quad \forall \ e_i\in\Delta_{d-2}(F), v\in\mathbb{P}_2(F).
\]
Taking $v = b_F^{\textrm{NC}} = 2 - d\sum^{d}_{i=1}\lambda^{2}_{i}\in\mathbb{P}_2(F)$, by \eqref{bTNCprop} we get $C_0 = 0$. By taking $v = \dfrac{d(d-1)}{2}\lambda_{i}^2-1\in\mathbb{P}_2(F)$ for $i = 1, 2, \ldots, d$, we obtain $C_1 = C_2 = \cdots = C_{d} = 0$. Hence the DoFs (\ref{p2f-dof1}) and (\ref{p2f-dof2}) are linearly independent.
The number of the DoFs \eqref{p2f-dof} therefore equals $\dim\mathbb P_2(F)$. If all these DoFs vanish, then $v\in\hat{\mathbb P}_2(F)$, and taking $q=v$ in \eqref{p2f-dof3} gives $v=0$.
%
% Obviously, the number of the DoFs \eqref{p2f-dof} is equal to the dimension of  $\mathbb{P}_2(F)$. Take $v\in \mathbb{P}_2(F)$ and assume all the DoFs \eqref{p2f-dof} vanish,
% then we prove $v=0$. We get from the vanishing DoFs (\ref{p2f-dof1})-(\ref{p2f-dof2}) that $v\in\hat{\mathbb{P}}_2(F)$. Taking $q=v$ in the vanishing DoFs (\ref{p2f-dof3}), then we conclude $v = 0$.
\end{proof}

\begin{lemma}\label{lem:ncfmQunisol}
For a $d$-dimensional simplex $T$,
the DoFs
\begin{subequations}\label{dof-q}
\begin{align}
\label{dof1-q}
&\int_{F} q\ds, \quad F\in\Delta_{d-1}(T),\\
\label{dof2-q}
&\int_{T}q\dx
\end{align}
\end{subequations}
are	unisolvent for the space $\mathbb{P}_{1}(T)\oplus \textrm{span}\{b_T^{\mathrm{NC}}\}$.
\end{lemma}
\begin{proof}
First both the number of DoFs \eqref{dof-q} and the dimension of space $\mathbb{P}_{1}(T)\oplus \textrm{span}\{b_T^{\textrm{NC}}\}$ are $d+2$. Assume that $q=q_1+q_2\in\mathbb{P}_{1}(T)\oplus \textrm{span}\{b_T^{\textrm{NC}}\}$ and all the DoFs \eqref{dof-q} vanish, where $q_1\in\mathbb{P}_{1}(T)$ and $q_2\in\textrm{span}\{b_T^{\textrm{NC}}\}$. It holds from \eqref{bTNCprop} and the vanishing DoF \eqref{dof1-q} that $q_1=0$. Then $q=0$ follows from the vanishing DoF~\eqref{dof2-q}.
\end{proof}

\begin{lemma}\label{lem:P2divbubble}
Let $\boldsymbol{v}\in \mathbb{P}_{2}(T;\mathbb{R}^{d})\cap H_0(\div,T)$ satisfy that $\div\boldsymbol{v}=0$ and DoF \eqref{dof6} vanishes. Then $\boldsymbol{v}=0$. 
\end{lemma}
\begin{proof}
We can express (cf. \cite[Section 3.3]{MR4458899})
\begin{align}\label{v-expression}
\boldsymbol{v} = \sum_{0\leq i<j\leq d}2a_{ij}\lambda_{i}\lambda_j\boldsymbol{t}_{i,j} = \sum^d_{i,j=0}a_{ij}\lambda_{i}\lambda_j\boldsymbol{t}_{i,j},
\end{align}
where $A=(a_{ij})_{0\leq i,j\leq d}$ is a $(d+1)\times(d+1)$ skew-symmetric matrix.
Since
\begin{equation*}
\div\boldsymbol{v} = \sum^d_{i,j=0}a_{ij}(\lambda_{i}-\lambda_j) = 2\sum^d_{i,j=0}a_{ij}\lambda_{i}, 
\end{equation*}
it follows from $\div\boldsymbol{v}=0$ that
\begin{align}\label{dof-proof2}
\sum^d_{j=0}a_{ij}=0, \quad  i = 0, 1, \ldots, d.
\end{align}
% Then (\ref{dof-proof2}) means that there is a non-zero skew-symmetric matrix $A$ such that (\ref{dof-proof1}) holds. 
Noting that $\nabla\boldsymbol{v}\in\mathbb{P}_{1}(T;\mathbb{R}^{d\times d})$,
the vertex average of $\skw\nabla\boldsymbol{v}$ coincides with its element average, namely
\begin{equation}\label{vertex-average}
\dfrac{1}{d+1}\sum_{i=0}^{d}(\skw\nabla\boldsymbol{v})(\texttt{v}_{i})
=
\dfrac{1}{|T|}\int_{T}\skw\nabla\boldsymbol{v} \dx.
\end{equation}
Combining \eqref{v-expression} and \eqref{vertex-average} gives
%By the fact that $\nabla\boldsymbol{v}\in\mathbb{P}_{1}(T;\mathbb{R}^{d\times d})$, we have
\begin{align*}%\label{dof-proof3}
\dfrac{1}{d+1}\sum_{i=0}^{d}(\skw\nabla\boldsymbol{v})(\texttt{v}_{i}) = \dfrac{1}{|T|}\sum^d_{i,j=0}a_{ij}\int_{T}\skw(\nabla(\lambda_{i}\lambda_j)\otimes\boldsymbol{t}_{i,j}) \dx.
\end{align*}
%which together with the vanishing DoFs (\ref{dof6}) implies that
%\begin{align*}
%\sum_{i,j=0}^{d}a_{ij}\skw(\nabla\lambda_j\otimes\boldsymbol{t}_{i,j}+\nabla\lambda_i\otimes\boldsymbol{t}_{i,j})=0.
%\end{align*}
Since
\begin{equation*}
\nabla(\lambda_i\lambda_j)=\lambda_i\nabla\lambda_j+\lambda_j\nabla\lambda_i,
\quad\text{and}\quad
\frac{1}{|T|}\int_T\lambda_i\,dx =\frac{1}{|T|}\int_T\lambda_j\,dx =\frac{1}{d+1},
\end{equation*}
the vanishing DoF \eqref{dof6} implies that
\begin{align*}
\sum_{i,j=0}^{d}a_{ij}\skw(\nabla\lambda_j\otimes\boldsymbol{t}_{i,j}+\nabla\lambda_i\otimes\boldsymbol{t}_{i,j})=0.
\end{align*}
Due to the skew-symmetry of $a_{ij}$ and $\boldsymbol{t}_{i,j}$, we have 
\begin{align*}%\label{dof-proof4}
\sum_{i,j=0}^{d}a_{ij}\skw(\nabla\lambda_i\otimes\boldsymbol{t}_{i,j})=0.
\end{align*}
Thanks to \eqref{dof-proof2}, multiply $\nabla\lambda_{k}$ on the right hand side of the last equation to get
\begin{align*}%\label{dof-proof5}
\sum_{i=0}^{d}a_{ik}\nabla\lambda_{i}-\sum_{i=0}^{d}\sum_{j=0}^{d}a_{ij}\boldsymbol{t}_{i,j}(\nabla\lambda_{i}\cdot\nabla\lambda_{k}) = 0,\quad k = 0,1,\ldots,d.
\end{align*}
The dot product of $\nabla\lambda_{\ell}$ with the last equation yields
%Again dot multiply $\nabla\lambda_{i}$ ($i = 0,1,\ldots,d$) on both sides of (\ref{dof-proof5}), we obtain
\begin{align}\label{dof-proof6}
\sum_{i=0}^{d}a_{ik}(\nabla\lambda_{i}\cdot\nabla\lambda_{\ell})-\sum_{i=0}^{d}a_{i\ell}(\nabla\lambda_{i}\cdot\nabla\lambda_{k}) = 0,\quad k,\ell = 0,1,\ldots,d.
\end{align}
Let $B=(\nabla\lambda_{i}\cdot\nabla\lambda_{j})_{0\leq i,j\leq d}$, which is symmetric positive semi-definite.
Noting that $\nabla\lambda_{0} = -(\nabla\lambda_{1}+\cdots+\nabla\lambda_{d})$, we have ${\rm rank}(B) = d$. Then \eqref{dof-proof6} becomes
\begin{equation*}
AB+BA=0.
\end{equation*}
Since $B$ is symmetric positive semi-definite, it can be expressed in the form $B = M^{\intercal}DM$, where $M$ is an orthogonal matrix and $D$ is a symmetric positive semi-definite diagonal matrix with ${\rm rank}(D) = d$. Hence
\begin{align*}
(MAM^{\intercal})D + D(MAM^{\intercal})= 0,
\end{align*}
which combined with the fact $MAM^{\intercal}$  is skew-symmetric yields $MAM^{\intercal}=0$. Thus $A=0$ and $\boldsymbol{v}=0$.
\end{proof}

\begin{lemma}
\label{dof-uni-solvent}
For a $d$-dimensional simplex $T$, DoFs \eqref{dof} are unisolvent for $V(T)$.
\end{lemma}
\begin{proof}
The total number of DoFs \eqref{dof3}--\eqref{dof5} is $d^2(d+1)$. Hence, the total number of DoFs \eqref{dof} is
\[
\dfrac{d^2(d+1)}{2}+\dfrac{(d-1)d(d+1)}{2}+d^2(d+1)+\dfrac{(d-1)d}{2}=d(2d^2+2d-1).\]
Take $\boldsymbol{v}\in  V(T)$ and assume all the DoFs \eqref{dof} vanish,
then we prove $\boldsymbol{v}=0$.
By the vanishing DoFs (\ref{dof1})-(\ref{dof2}) and Lemma \ref{p2f-lemma}, 
we get $\boldsymbol{v}\cdot\boldsymbol{n}|_{\partial T}=0$. Hence, $\div\boldsymbol{v}\in L_0^2(T)$.
Together with the vanishing DoF (\ref{dof4}),
we obtain from \eqref{eq:divVT} and Lemma~\ref{lem:ncfmQunisol} that $\div\boldsymbol{v}=0$. Then we can write $\boldsymbol{v}=\boldsymbol{v}_1+\boldsymbol{v}_2+\boldsymbol{v}_3+\boldsymbol{v}_4$ with $\boldsymbol{v}_1\in\mathbb{P}_{2}(T;\mathbb{R}^{d})$ satisfying $\div\boldsymbol{v}_1=0$, $\boldsymbol{v}_2\in\sum_{i=0}^{d}\sum_{j=1}^{d-1}\div\big(\skw(\boldsymbol{n}_{F_i}\otimes\boldsymbol{t}_j^{F_i})b_{T}b_{F_i}\mathbb{P}_{1}(F_i)\big)$, $\boldsymbol{v}_3\in\sum_{i=0}^{d}\sum_{j=1}^{d-1}\div\big(\skw(\boldsymbol{n}_{F_i}\otimes\boldsymbol{t}_j^{F_i})b^2_{T}b_{F_i}\mathbb{P}_{0}(F_i)\big)$, and
\[
\boldsymbol{v}_4\in\sum_{e\in\Delta_{d-2}(T)}\sum_{k=1}^{d-2}\div\big(\skw(\boldsymbol m_e\otimes\boldsymbol t_k^e)b_e^2\mathbb P_0(e)\big).
\]

%Noting that $\boldsymbol{v}_2\in H_0(\div,T)$ and $\boldsymbol{v}_3\in H_0^1(T;\mathbb R^d)$, we have $\boldsymbol{v}_1\in \mathbb{P}_{2}(T;\mathbb{R}^{d})\cap H_0(\div,T)$ and DoF \eqref{dof6} for $\boldsymbol{v}_1$ vanish. 
Noting that $\boldsymbol{v}_2,\boldsymbol{v}_4\in H_0(\div,T)$ and $\boldsymbol{v}_3\in H_0^1(T;\mathbb R^d)$, we have $\boldsymbol{v}_1\in \mathbb{P}_{2}(T;\mathbb{R}^{d})\cap H_0(\div,T)$.
Since the bubble functions and their first-order derivatives vanish at all vertices, $\boldsymbol v_1$ has vanishing DoF \eqref{dof6}.
Apply Lemma~\ref{lem:P2divbubble} to conclude $\boldsymbol{v}_1=0$.

Now
\begin{align*}
\boldsymbol{v}&=\sum_{i=0}^{d}\sum_{j=1}^{d-1}\div(\skw(\boldsymbol{n}_{F_i}\otimes\boldsymbol{t}_j^{F_i})b_{T}b_{F_i}p_{ij})
+\sum_{i=0}^{d}\sum_{j=1}^{d-1}\div(\skw(\boldsymbol{n}_{F_i}\otimes\boldsymbol{t}_j^{F_i})b^2_{T}b_{F_i}q_{ij})\\
&\quad+\sum_{e\in\Delta_{d-2}(T)}\sum_{k=1}^{d-2}\div\big(\skw(\boldsymbol m_e\otimes\boldsymbol t_k^e)b_e^2w_{e,k}\big),
\end{align*}
where $p_{ij}\in\mathbb{P}_1(F_i)$, $q_{ij}\in\mathbb{P}_0(F_i)$ and $w_{e,k}\in\mathbb{P}_0(e)$.
Noting that $\boldsymbol m_e|_e=\boldsymbol 0$, we have
\[
\div\big(\skw(\boldsymbol m_e\otimes\boldsymbol t_k^e)b_e^2w_{e,k}\big)\big|_e
=-\frac12(\div\boldsymbol m_e)b_e^2w_{e,k}\boldsymbol t_k^e.
\]
Since the face bubbles and the bubbles associated with other
$(d-2)$-dimensional subsimplices vanish on $e$, it follows from
the orthonormality of $\{\boldsymbol t_k^e\}_{k=1}^{d-2}$ that
\[
(\boldsymbol v\cdot\boldsymbol t_k^e)|_e=-\frac12(\div\boldsymbol m_e)b_e^2w_{e,k}.
\]
Since
\[
\begin{aligned}
\div\boldsymbol m_e
&=\boldsymbol n_{F_j,e}\cdot\nabla\lambda_i+\boldsymbol n_{F_i,e}\cdot\nabla\lambda_j =-|\nabla_{F_j}\lambda_i|-|\nabla_{F_i}\lambda_j|<0
\end{aligned}
\]
for $e=e_{ij}$, the vanishing DoF \eqref{dof1} implies $w_{e,k}=0$.
By $(\boldsymbol{v}\cdot\boldsymbol{t}_j^{F_i})|_{F_i}=-\frac{1}{2}(\boldsymbol{n}_{F_i}\cdot\nabla\lambda_i)b_{F_i}^2p_{ij}$, we obtain $p_{ij}=0$ from the vanishing DoF \eqref{dof3}.
By $\partial_n(\boldsymbol{v}\cdot\boldsymbol{t}_j^{F_i})|_{F_i}=-(\boldsymbol{n}_{F_i}\cdot\nabla\lambda_i)(\partial_nb_T)b_{F_i}^2q_{ij}$, we conclude $q_{ij}=0$ from the vanishing DoF \eqref{dof5}.
Thus, $\boldsymbol{v}=0$.
This argument also shows that a vanishing linear combination of the functions in the defining sum of $V(T)$ has all its coefficients equal to zero. Thus, the sum defining $V(T)$ is direct, and its dimension equals the number of the DoFs \eqref{dof}. This completes the proof of unisolvence.
\end{proof}

Define the global $H^2$-nonconforming finite element space
\begin{align*}
V_{h}&=\{\boldsymbol{v}_{h}\in L^{2}(\Omega;\mathbb{R}^{d}):\boldsymbol{v}_h|_T\in V(T)~\textrm{for}~T\in\mathcal{T}_h; \textrm{all~the~DoFs \eqref{dof1}-\eqref{dof5}} \\
&\qquad\qquad\;\;\textrm{are~single-valued, and DoFs \eqref{dof1}-\eqref{dof3} on boundary vanish}\}. %,\\
% V_{h0}&=\{\boldsymbol{v}_{h}\in  V_{h}: \textrm{DoFs \eqref{dof4}-\eqref{dof5} on boundary vanish}\}.
\end{align*}
Note that DoF~\eqref{dof6} is treated as interior to $T$, and therefore imposes no continuity constraints.
\begin{remark}\rm
Notice that
\begin{align*}
\int_{F}\div\boldsymbol{v}\,\dd s&=\int_{F}\partial_n(\boldsymbol{v}\cdot\boldsymbol{n})\,\dd s+\int_{F}\div_F\boldsymbol{v}\,\dd s \\
&=\int_{F}\partial_n(\boldsymbol{v}\cdot\boldsymbol{n})\,\dd s+\sum_{e\in\partial F}\int_{e}\boldsymbol{v}\cdot\boldsymbol{n}_{F,e}\,\dd s.
\end{align*}
With DoF \eqref{dof1}, when defining the finite element space $V_h$, DoF \eqref{dof4} can be equivalently replaced with
\begin{equation}\label{dof4new}
\int_{F}\partial_n(\boldsymbol{v}\cdot\boldsymbol{n})\,\dd s,  \quad F\in\Delta_{d-1}(T).
\end{equation}
\end{remark}
Clearly $V_{h}\subset H_0(\div,\Omega)$, but $V_{h}\not\subseteq H_0^1(\Omega;\mathbb R^d)$.
The finite element space $V_h$ %and $V_{h0}$ 
has the weak continuity
\begin{align}
\label{weak-continuity-vh1}
&\int_{F}[\![\nabla_h\boldsymbol{v}_h]\!]\ds=0 \quad \forall \ \boldsymbol{v}_h\in V_h, \ F\in\mathring{\mathcal{F}}_h,\\
\label{weak-continuity-vh2}
&\int_{F}[\![\boldsymbol{v}_h]\!]\cdot \boldsymbol{q} \ds=0 \quad \forall \ \boldsymbol{v}_h\in V_h, \ \boldsymbol{q}\in\mathbb{P}_1(F;\mathbb{R}^d), \ F\in\mathcal{F}_h.
% ,\\
% \label{weak-continuity-v0h1}
% &\int_{F}[\![\nabla_h\boldsymbol{v}_h]\!]\ds=0 \quad \forall \ \boldsymbol{v}_h\in V_{h0}, \ F\in\mathcal{F}_h.
\end{align}
The continuity of the mean normal derivative follows from the single-valuedness of the DoFs \eqref{dof5} and \eqref{dof4new}, while that of the mean tangential derivatives follows from the single-valuedness of the DoF \eqref{dof1} and integration by parts on each face. These together yield \eqref{weak-continuity-vh1}. The weak continuity \eqref{weak-continuity-vh2} follows from the inclusion $V_h\subset H_0(\div,\Omega)$ together with the single-valuedness of the DoF \eqref{dof3}.
% \begin{comment}
% According to \cite{ming2006morley}, it holds
% \begin{align*}
% &\int_{F}[\![\nabla_h\boldsymbol{v}_h]\!]ds=0 \quad \forall \ \boldsymbol{v}_h\in  V_h, F\in\mathcal{F}^i(\mathcal{T}_h),\\
% &\int_{F}[\![\nabla_h\boldsymbol{v}_h]\!]ds=0 \quad \forall \ \boldsymbol{v}_h\in V_{h0}, \ F\in\mathcal{F}(\mathcal{T}_h).
% \end{align*}
% \end{comment}

Introduce the following finite element space for $\div\boldsymbol{u}$
\begin{align*}
Q_{h}=\{q_{h}\in L_0^{2}(\Omega)&: q_h|_T\in \mathbb{P}_{1}(T)\oplus \textrm{span}\{b_T^{\textrm{NC}}\}~\textrm{for}~T\in\mathcal{T}_h, \\
&\qquad\quad\;\;\textrm{ and DoF \eqref{dof1-q} is single-valued}\}. %,\\
% Q_{h0}=\{&q_{h}\in Q_{h}:~\textrm{DoFs}~(\ref{dof1-q})~\textrm{on~boundary~vanish}\}.
\end{align*}
Space $Q_{h}$ is $H^1$-nonconforming.

\subsection{Interpolation operators}\label{NCFEM-Q}
For the locking-free analysis, we need to define some interpolation operators.
We first introduce the quadratic
Brezzi-Douglas-Marini (BDM) element \cite{MR4458899,ChenChenHuangWei2024,BrezziDouglasMarini1986,BrezziDouglasDuranFortin1987,nedelec1986new}. The quadratic BDM element takes $\mathbb{P}_{2}(T;\mathbb{R}^{d})$ as the shape function space, and the DoFs are chosen as
\begin{align}
\label{BDM-dof1}
\int_{F}\boldsymbol{v}\cdot \boldsymbol n \, q\ds, & \quad q\in \mathbb{P}_{2}(F), \, F\in\Delta_{d-1}(T),\\
\label{BDM-dof2}
\int_{T}\boldsymbol{v}\cdot \boldsymbol{q}\dx, & \quad \boldsymbol{q}\in 
\mathbb{P}_{0}(T;\mathbb R^d)\oplus\mathbb{P}_{0}(T;\mathbb{K})\boldsymbol{x}.
%\mathbb{P}_1(T;\mathbb{R}^d)~\textrm{satisfying}~ \boldsymbol{x}\cdot\boldsymbol{q}\in\mathbb{P}_1(T).
\end{align}
Let $I_T^{\rm BDM}:H^1(T;\mathbb{R}^{d})\rightarrow\mathbb{P}_{2}(T;\mathbb{R}^{d})$ be the nodal interpolation operator based on DoFs (\ref{BDM-dof1})-(\ref{BDM-dof2}). It holds \cite{MR3097958}
\begin{align}
\label{BDM-div}
\div(I_T^{\rm BDM}\boldsymbol{v})=Q_{1,T}(\div\boldsymbol{v}) \quad\forall~\boldsymbol{v}\in H^1(T;\mathbb{R}^{d}).
\end{align}
% Let $I_h^{\rm BDM}$ be the global version of $I_T^{\rm BDM}$.
Now we define the interpolation operator $I_h^V: H^1_0(\Omega;\mathbb{R}^d)\rightarrow V_h$ as follows:
for any $\boldsymbol{v}\in H^1_0(\Omega;\mathbb{R}^d),$
\begin{align}
\notag%\label{interpolation-I1}
&\int_e I_h^V\boldsymbol{v} \ds=\dfrac{1}{\#\mathcal{T}_e}\sum_{T\in\mathcal{T}_e}\int_eI_T^{\rm BDM}\boldsymbol{v}\ds \quad \forall~e\in\mathring{\mathcal{E}}_h,\\
\label{interpolation-I2}&\int_{F}(I_h^V\boldsymbol{v})\cdot\boldsymbol{n} \, q\ds=\int_{F}\boldsymbol{v}\cdot\boldsymbol{n} \, q\ds \quad \forall~q\in \mathbb{P}_0(F)\oplus\hat{\mathbb{P}}_2(F), F\in\mathring{\mathcal{F}}_h,\\
\notag% \label{interpolation-I3}
&\int_{F}\Pi_F(I_h^V\boldsymbol{v}) \cdot \boldsymbol{q} \ds=\int_{F}\Pi_F\boldsymbol{v} \cdot \boldsymbol{q}\ds\quad \forall~\boldsymbol{q}\in\mathbb{P}_1(F;\mathscr T^F), F\in\mathring{\mathcal{F}}_h,\\
\label{interpolation-I4}&\int_{F}\div(I_h^V\boldsymbol{v})\ds=\dfrac{1}{\#\mathcal{T}_F}\sum_{T\in\mathcal{T}_F}\int_{F}\div(I_T^{\rm BDM}\boldsymbol{v})\ds \quad \forall \ F\in\mathcal{F}_h,\\
\notag% \label{interpolation-I5}
&\int_{F}\partial_{n}(\Pi_F(I_h^V\boldsymbol{v}))\ds=\dfrac{1}{\#\mathcal{T}_F}\sum_{T\in\mathcal{T}_F}\int_{F}\partial_{n}(\Pi_F(I_T^{\rm BDM}\boldsymbol{v}))\ds \quad \forall \ F\in\mathcal{F}_h,\\
\notag% \label{interpolation-I6}
&\frac 1{d+1}\sum_{i=0}^{d}(\skw\nabla(I_h^V\boldsymbol{v}))(\texttt{v}_{i})=\frac 1{d+1}\sum_{i=0}^{d}(\skw\nabla(I_T^{\rm BDM}\boldsymbol{v}))(\texttt{v}_{i}) \quad \forall \ T\in\mathcal{T}_h.
\end{align}
% Meanwhile, we can construct another interpolation operator
% $$I_{h0}:H^1_0(\Omega;\mathbb{R}^d)\rightarrow V_{h0}$$
% by forcing (\ref{interpolation-I4}) and (\ref{interpolation-I5}) on boundary to be zero.
\begin{lemma}
\label{interpolation-I-lemma1}
For integers $0\leq m\leq 3$ and $j=0,1$, we have for $\boldsymbol{v}\in H^1_0(\Omega;\mathbb{R}^d)\cap H^s(\Omega;\mathbb{R}^d)$ that
\begin{align}
\label{Ih-error1}
|\boldsymbol{v}-I_h^V\boldsymbol{v}|_{m,T}&\lesssim h_T^{s-m}|\boldsymbol{v}|_{s,\omega_T} \qquad\; \textrm{ for } \max\{m,1\}\leq s\leq 3, \\
\label{Ih-error11}
\|\partial_n^j\boldsymbol{\varepsilon}(\boldsymbol{v}-I_h^V\boldsymbol{v})\|_{0,\partial T}&\lesssim h_T^{s-j-3/2}|\boldsymbol{v}|_{s,\omega_T} \quad \textrm{ for } j+2\leq s\leq 3.
\end{align}
\end{lemma}
\begin{proof}
Let $\boldsymbol{w}=(I_h^V\boldsymbol{v})|_T-I_T^{\rm BDM}\boldsymbol{v}\in V(T)$ for simplicity. Note that the DoFs \eqref{dof2} and \eqref{dof6} of $\boldsymbol{w}$ vanish, while the DoFs \eqref{dof4} and \eqref{dof5} vanish on boundary faces.
By the unisolvence established in Lemma~\ref{dof-uni-solvent}, scaling arguments, and the definition of $I_h^V$, we obtain
\begin{align*}
\|\boldsymbol{w}\|^2_{0,T}& \lesssim h^2_T\sum_{e\in\Delta_{d-2}(T)}\|Q_{0,e}\boldsymbol{w}\|^2_{0,e} + h_T\sum_{F\in\Delta_{d-1}(T)}\|Q_{1,F}(\Pi_F\boldsymbol{w})\|^2_{0,F}\\
&\quad+h^3_T\sum_{F\in\Delta_{d-1}(T)\cap\mathring{\mathcal{F}}_h}\big(\|Q_{0,F}(\div\boldsymbol{w})\|^2_{0,F} + \|Q_{0,F}(\partial_{n}(\Pi_F\boldsymbol{w}))\|^2_{0,F}\big)\\
&\lesssim h^2_T\sum_{e\in\Delta_{d-2}(T)\cap\mathring{\mathcal{E}}_h}\sum_{T^*\in\mathcal{T}_e}\|I_T^{\rm BDM}\boldsymbol{v}-I_{T^*}^{\rm BDM}\boldsymbol{v}\|^2_{0,e}\\
&\quad+h^2_T\sum_{e\in\Delta_{d-2}(T)\cap\mathcal{E}^\partial_h}\|I_T^{\rm BDM}\boldsymbol{v}\|^2_{0,e}+h_T\sum_{F\in\Delta_{d-1}(T)}\|\boldsymbol{v}-I_T^{\rm BDM}\boldsymbol{v}\|^2_{0,F}\\
&\quad +h^3_T\sum_{F\in\Delta_{d-1}(T)\cap\mathring{\mathcal{F}}_h}\big(\|[\![\div(I_T^{\rm BDM}\boldsymbol{v})]\!]\|^2_{0,F}+\|[\![\partial_{n}(\Pi_F(I_T^{\rm BDM}\boldsymbol{v}))]\!]\|^2_{0,F}\big).
\end{align*}
%\begin{align*}
%	\|\boldsymbol{w}\|^2_{0,T}& \lesssim h^2_T\sum_{e\in\Delta_{d-2}(T)}\sum_{i=1}^2\|Q_{0,e}(\boldsymbol{w}\cdot\boldsymbol{n}_i)\|^2_{0,e} + h_T\sum_{F\in\Delta_{d-1}(T)}\|Q_{1,F}(\Pi_F\boldsymbol{w})\|^2_{0,F}\\
%	&\quad+h^3_T\sum_{F\in\Delta_{d-1}(T)}\|Q_{0,F}(\div\boldsymbol{w})\|^2_{0,F} + h^3_T\sum_{F\in\Delta_{d-1}(T)}\|Q_{0,F}(\partial_{n}(\Pi_F\boldsymbol{w}))\|^2_{0,F}\\
%	&\lesssim h_T\sum_{e\in\Delta_{d-2}(T)}\sum_{F\in\mathcal{F}_h, e\in\partial F}\|[\![I_T^{\rm BDM}\boldsymbol{v}]\!]\|^2_{0,F}+h_T\sum_{F\in\Delta_{d-1}(T)}\|\boldsymbol{v}-I_T^{\rm BDM}\boldsymbol{v}\|^2_{0,F}\\
%	&\quad +h^3_T\sum_{F\in\Delta_{d-1}(T)\cap\mathring{\mathcal{F}}_h}\|[\![\nabla(I_T^{\rm BDM}\boldsymbol{v})]\!]\|^2_{0,F}.
%\end{align*}
Using element chains around each $(d-2)$-dimensional subsimplex (ending at a boundary face for boundary subsimplices), the trace inequality and the inverse inequality, we obtain
\begin{align*}
\|\boldsymbol{w}\|^2_{0,T}&\lesssim h_T\sum_{e\in\Delta_{d-2}(T)}\sum_{F\in\mathcal{F}_h, e\in\partial F}\|[\![I_T^{\rm BDM}\boldsymbol{v}]\!]\|^2_{0,F}+h_T\sum_{F\in\Delta_{d-1}(T)}\|\boldsymbol{v}-I_T^{\rm BDM}\boldsymbol{v}\|^2_{0,F}\\
&\quad +h^3_T\sum_{F\in\Delta_{d-1}(T)\cap\mathring{\mathcal{F}}_h}\|[\![\nabla(I_T^{\rm BDM}\boldsymbol{v})]\!]\|^2_{0,F}.
\end{align*}
Apply the inverse inequality and the interpolation error estimates of $I_T^{\rm BDM}$ to get
\begin{equation*}
h^{m}_T|I_h^V\boldsymbol{v}-I_T^{\rm BDM}\boldsymbol{v}|_{m,T}\lesssim \|\boldsymbol{w}\|_{0,T}\lesssim h_T^{s}|\boldsymbol{v}|_{s,\omega_T}.
\end{equation*}
Then we derive the estimate (\ref{Ih-error1}) from the triangle inequality and the interpolation error estimate of $I_T^{\rm BDM}$.

By the trace inequality,
\begin{equation*}
\|\partial_n^j\boldsymbol{\varepsilon}(\boldsymbol{v}-I_h^V\boldsymbol{v})\|_{0,\partial T}\lesssim h_T^{-1/2}|\boldsymbol{\varepsilon}(\boldsymbol{v}-I_h^V\boldsymbol{v})|_{j,T} + h_T^{1/2}|\boldsymbol{\varepsilon}(\boldsymbol{v}-I_h^V\boldsymbol{v})|_{j+1,T}.
\end{equation*}
Thus, \eqref{Ih-error11} holds from \eqref{Ih-error1}.
\end{proof}

% \begin{comment}
% \begin{align*}
% &Q_{h}=\{q_{h}\in L^{2}(\Omega):q_h|_T\in Q(T)~\textrm{for~each}~T\in\mathcal{T}_h;
% \int_{F}[\![q_h]\!]ds=0~\textrm{for~each}~F\in\mathring{\mathcal{F}}_h\},\\
% &Q_{h0}=\{q_{h}\in Q_{h}:\int_{F}q_hds=0~\textrm{for~each}~F\in\mathcal F_h^{\partial}\}.
% \end{align*}
% It is easy to see that $\div V_h=Q_h$ and $\div V_{h0}=Q_{h0}$.
% \end{comment}
Then we define an interpolation operator $I_h^Q: L_0^2(\Omega)\rightarrow Q_{h}$ as follows:
for any $q\in L_0^2(\Omega)$,
\begin{align}
\label{interpolation-Jh1}\int_{F}I_h^Q q\ds&=\dfrac{1}{\#\mathcal{T}_F}\sum_{T\in\mathcal{T}_F}\int_{F}Q_{1,T} q\ds \quad \forall \ F\in\mathcal{F}_h,\\
\label{interpolation-Jh2}\int_{T}I_h^Q q\dx&=\int_{T}q\dx \quad\forall \ T\in\mathcal{T}_h.
\end{align}
% Through forcing (\ref{interpolation-Jh1}) on boundary to be zero, we can define another interpolation
% operator $J_{h0}:L_0^2(\Omega)\rightarrow Q_{h0}$.
Utilizing the definition of above interpolation operators, we establish the following commutative property.
\begin{lemma}
\label{commutative}
It holds
\begin{align}
\label{commutative1}
\div (I_h^V\boldsymbol{v})&=I_h^Q(\div\boldsymbol{v}) \quad \forall~\boldsymbol{v}\in H^1_0(\Omega;\mathbb{R}^d). %,\\
% \label{commutative2}
% \div I_{h0}\boldsymbol{v}&=J_{h0}\div\boldsymbol{v} \quad \forall \ \boldsymbol{v}\in H^1_0(\Omega;\mathbb{R}^d).
\end{align}
\end{lemma}
\begin{proof}
For $\boldsymbol{v}\in H^1_0(\Omega;\mathbb{R}^d)$, we have $\div\boldsymbol{v}\in L^2_0(\Omega)$. 
By (\ref{interpolation-I4}), (\ref{BDM-div}) and (\ref{interpolation-Jh1}), %it's easy to see
\[
\int_{F}\div(I_h^V\boldsymbol{v})\ds=\dfrac{1}{\#\mathcal{T}_F}\sum_{T\in\mathcal{T}_F}\int_{F}Q_{1,T}(\div\boldsymbol{v})\ds=\int_{F}I_h^Q(\div\boldsymbol{v})\ds \quad \forall~F\in\mathcal{F}_h.
\]
Using (\ref{interpolation-Jh2}), integration by parts, and (\ref{interpolation-I2}),
\begin{equation*}
\int_{T}(\div(I_h^V\boldsymbol{v})-I_h^Q(\div\boldsymbol{v}))\dx=\int_{T}\div(I_h^V\boldsymbol{v}-\boldsymbol{v})\dx=0\quad \forall~T\in\mathcal{T}_h.
\end{equation*}
Hence (\ref{commutative1}) holds from the last two equations and Lemma~\ref{lem:ncfmQunisol}.
\end{proof}

\begin{corollary}
The following div-surjectivity holds
\begin{equation}\label{divontodiscrete}
\div V_h=Q_h.	
\end{equation}
\end{corollary}
\begin{proof}
By \eqref{eq:divVT} and the definition of $V_h$, $\div V_h\subseteq Q_h$.
For $q_h\in Q_h\subset L_0^2(\Omega)$, there exists a $\boldsymbol{v}\in H^1_0(\Omega;\mathbb{R}^d)$ such that $\div\boldsymbol{v}=q_h$.
We define $\boldsymbol v_h\in V_h$ by taking
all its DoFs to be those of $I_h^V\boldsymbol v$, except for the face
divergence moments in \eqref{interpolation-I4}, which we prescribe by
\[
\int_{F}\div\boldsymbol{v}_h\ds=\int_{F}q_h\ds \quad\forall\ F\in\mathcal F_h.
\]
Applying integration by parts, and (\ref{interpolation-I2}),
\[
\int_{T}(\div\boldsymbol{v}_h-q_h)\dx=\int_{T}(\div\boldsymbol{v}-q_h)\dx=0\quad \forall~T\in\mathcal{T}_h.
\]
By Lemma~\ref{lem:ncfmQunisol}, we get $\div\boldsymbol v_h=q_h$, which proves $Q_h\subseteq\div V_h$.
%Then $I_h^V\boldsymbol{v}\in V_h$, and by \eqref{commutative1} we have $\div(I_h^V\boldsymbol{v})=I_h^Q(\div\boldsymbol{v})=I_h^Qq_h=q_h$.
\end{proof}

\begin{remark}\rm
The div-surjectivity~\eqref{divontodiscrete} is equivalent to the exactness of the finite element complex
\[
V_h \xrightarrow{\div} Q_h \xrightarrow{} 0.
\]
This complex may be viewed as a nonconforming discretization of the continuous complex
\[
H^2(\Omega;\mathbb{R}^d)\cap H_0^1(\Omega;\mathbb{R}^d)
\xrightarrow{\div}
H^1(\Omega)\cap L_0^2(\Omega)
\xrightarrow{} 0.
\]
In contrast to the discrete complex, the continuous complex is not exact in general; in the two-dimensional polygonal case, see \cite[Theorem~3.1]{ArnoldScottVogelius1988}. This distinction provides one motivation for using a nonconforming discretization.
\end{remark}

\begin{lemma}
\label{interpolation-Jh-lemma}
For integers $0\leq m\leq 2$ and $j=0,1$, we have for any $\boldsymbol{v}\in H^1_0(\Omega;\mathbb{R}^d)$ satisfying $\div\boldsymbol{v}\in H^{s}(\Omega)$ that
\begin{align}
\label{Jh-error1}
|\div\boldsymbol{v}-\div(I_h^V\boldsymbol{v})|_{m,T}&\lesssim h_T^{s-m}|\div\boldsymbol{v}|_{s,\omega_T}
\qquad\; \textrm{ for } m\leq s\leq 2, \\
\label{Jh-error11}
\|\partial_n^j\div(\boldsymbol{v}-I_h^V\boldsymbol{v})\|_{0,\partial T}&\lesssim h_T^{s-j-1/2}|\div\boldsymbol{v}|_{s,\omega_T}
\quad \textrm{ for } j+1\leq s\leq 2. 
\end{align}
\end{lemma}
\begin{proof}
Let $q = \div\boldsymbol{v}\in L^2_0(\Omega)\cap H^{s}(\Omega)$, and $w=(I_h^Q q)|_T-Q_{1,T}q\in \mathbb{P}_{1}(T)\oplus \textrm{span}\{b_T^{\mathrm{NC}}\}$. Note that the DoF \eqref{dof2-q} of $w$ vanishes, while the DoF \eqref{dof1-q} vanishes on boundary faces. By scaling arguments, the definition of $I_h^Q$ and the trace inequality,
\begin{equation*}
\|w\|^2_{0,T} \lesssim h_T\sum_{F\in\Delta_{d-1}(T)\cap\mathring{\mathcal{F}}_h}\|Q_{0,F}w\|^2_{0,F}
\lesssim h_T\sum_{F\in\Delta_{d-1}(T)\cap\mathring{\mathcal{F}}_h}\|[\![Q_{1,T}q]\!]\|^2_{0,F}.
\end{equation*}
%Apply the similar argument for deriving Lemma \ref{interpolation-I-lemma1} to get
%\begin{equation*}
%|q-I_h^Q q|_{m,T}\lesssim h_T^{s-m}|q|_{s,\omega_T} \quad \forall \ q\in L^2_0(\Omega)\cap H^{s}(\Omega).
%\end{equation*}
Applying the inverse inequality and the error estimates of $Q_{1,T}$, we get
\begin{equation*}
h^{m}_T|I_h^Q q-Q_{1,T}q|_{m,T}\lesssim \|w\|_{0,T}\lesssim h_T^{s}|q|_{s,\omega_T},
\end{equation*}
which combined with the triangle inequality, the error estimate of $Q_{1,T}$ and \eqref{commutative1} yields \eqref{Jh-error1}.

By the trace inequality,
\begin{equation*}
\|\partial_n^j\div(\boldsymbol{v}-I_h^V\boldsymbol{v})\|_{0,\partial T}\lesssim h_T^{-1/2}|\div(\boldsymbol{v}-I_h^V\boldsymbol{v})|_{j,T} + h_T^{1/2}|\div(\boldsymbol{v}-I_h^V\boldsymbol{v})|_{j+1,T}.
\end{equation*}
Therefore, \eqref{Jh-error11} holds from \eqref{Jh-error1}.
\end{proof}

\begin{lemma}
\label{interpolation-I-lemma2}
Let $\boldsymbol{u}\in H_0^2(\Omega;\mathbb{R}^d)\cap H^3(\Omega;\mathbb{R}^d)$ be the solution of problem \eqref{SGE0}, and the solution $\boldsymbol{u}_0\in H^1_0(\Omega;\mathbb{R}^d)$ of problem \eqref{SGElinear} satisfy the regularity \eqref{elasregularity} with $2\leq s\leq 3$.
We have
\begin{align}
\label{Ih-error3}
\normmm{\boldsymbol{u}-I_h^V\boldsymbol{u}}_{\iota,\lambda,h} & \lesssim \iota h(|\boldsymbol{u}|_3+\sqrt{\lambda}|\div\boldsymbol{u}|_2)+ h(|\boldsymbol{u}|_2+\sqrt{\lambda}|\div\boldsymbol{u}|_1), \\ %(\iota h+h^2)|\boldsymbol{\sigma}(\boldsymbol{u})|_2
\label{Ih-error2}
\normmm{\boldsymbol{u}-I_h^V\boldsymbol{u}}_{\iota,\lambda,h} & \lesssim\iota^{1/2}\|\boldsymbol{f}\|_0+h^{s-1}\|\boldsymbol{f}\|_{s-2}.
\end{align}
\end{lemma}
\begin{proof}
Using (\ref{Ih-error1})-(\ref{Ih-error11}) and (\ref{Jh-error1})-(\ref{Jh-error11}), we acquire the estimate (\ref{Ih-error3}),
\begin{align}
\begin{split}
\label{interpolation-I-lemma2-pf}
\|\boldsymbol{\sigma}_h(\boldsymbol{u}-I_h^V\boldsymbol{u})\|_{0}&\leq \|\boldsymbol{\sigma}_h((\boldsymbol{u}-\boldsymbol{u}_0)-I_h^V(\boldsymbol{u}-\boldsymbol{u}_0))\|_{0}+\|\boldsymbol{\sigma}_h(\boldsymbol{u}_0-I_h^V\boldsymbol{u}_0)\|_{0} \\
&\lesssim |\boldsymbol{u}-\boldsymbol{u}_0|_1 + \lambda\|\div(\boldsymbol{u}-\boldsymbol{u}_0)\|_0 \\
&\quad +h^{s-1}|\boldsymbol{u}_0|_s+\lambda h^{s-1}|\div\boldsymbol{u}_0|_{s-1}, \\
\end{split}
\end{align}
and
\begin{align*}
\mathopen{\interleave}\boldsymbol{\sigma}_h(\boldsymbol{u}-I_h^V\boldsymbol{u})\mathclose{\interleave}_{1,h}
&\lesssim \mathopen{\interleave}\boldsymbol{\varepsilon}_h(\boldsymbol{u}-I_h^V\boldsymbol{u})\mathclose{\interleave}_{1,h} +
\lambda\mathopen{\interleave}\div(\boldsymbol{u}-I_h^V\boldsymbol{u})\mathclose{\interleave}_{1,h} \\
&\lesssim|\boldsymbol{u}|_2+\lambda|\div\boldsymbol{u}|_1.
\end{align*}
For $\lambda\geq0$, the stress--strain relation gives
\[
2\mu|\boldsymbol\varepsilon(\boldsymbol z)|^2+\lambda|\div\boldsymbol z|^2
\leq\frac1{2\mu}|\boldsymbol\sigma(\boldsymbol z)|^2.
\]
The same inequality holds for each derivative and boundary trace. Thus the last two estimates, together with \eqref{elasregularity}--\eqref{Regularity-divu}, imply \eqref{Ih-error2}, including when $\lambda=0$.
\end{proof}

\subsection{Nonconforming finite element complexes}
The finite element spaces $V_h$ and $Q_h$ constructed above can be applied to nonconforming finite element discretizations of the smooth Stokes complexes in two and three dimensions.

In two dimensions, after introducing a suitable $H^3$-nonconforming finite element space $W_h$, we obtain the exact discrete complex
\[
0\xrightarrow{\subset} W_h\xrightarrow{\curl}  V_h
\xrightarrow{\div} Q_h \xrightarrow {}0.
\]
In three dimensions, after introducing an $H(\nabla^2\curl)$-nonconforming finite element space $W_h$, we obtain the exact discrete complex
\[
0\xrightarrow{\subset}V_h^L\xrightarrow{\nabla} W_h \xrightarrow{\curl}  V_h
\xrightarrow{\div} Q_h \xrightarrow {}0.
\]
The definitions of spaces, the associated interpolation operators, and the proofs of exactness and commutativity are given in Appendix~\ref{app:complexes}.

%%%%%%%%%%%%%%%%%%%%%%%%%%%%%%%%%%%%%%%%
%% Nitsche Method 
%%%%%%%%%%%%%%%%%%%%%%%%%%%%%%%%%%%%%%%%
\section{Nonconforming finite element method}\label{sec4}
In this section, we will utilize the nonconforming finite elements constructed in the last section and Nitsche's technique to develop an optimal and robust nonconforming finite element method for the SGE model \eqref{SGE0}.

\subsection{Discrete formulation}
Through applying Nitsche's technique \cite{MR0341903,Schieweck2008,MR2917211}, i.e. imposing the boundary condition $\partial_{n}\boldsymbol{u} = 0$ weakly,
we propose the following nonconforming
finite element method for weak formulation \eqref{weak1}: find $\boldsymbol{u}_h\in V_h$ such that
\begin{align}\label{discrete-fem-weak}
\iota^2a_h(\boldsymbol{u}_h,\boldsymbol{v}_h)+b_h(\boldsymbol{u}_h,\boldsymbol{v}_h)=(\boldsymbol{f},\boldsymbol{v}_h) \quad \forall \ \boldsymbol{v}_h\in V_h,
\end{align} 
where 
\begin{align*}
a_h(\boldsymbol{u}_h,\boldsymbol{v}_h)&:=(\nabla_h\boldsymbol{\sigma}_h(\boldsymbol{u}_h),\nabla_h\boldsymbol{\varepsilon}_h(\boldsymbol{v}_h))-\sum_{F\in\mathcal F_h^{\partial}}(\boldsymbol{\sigma}_h(\boldsymbol{u}_h),\partial_{n}(\boldsymbol{\varepsilon}_h(\boldsymbol{v}_h)))_F\\
&\quad\;\;-\sum_{F\in\mathcal F_h^{\partial}}(\partial_{n}(\boldsymbol{\sigma}_h(\boldsymbol{u}_h)),\boldsymbol{\varepsilon}_h(\boldsymbol{v}_h))_F + \eta\sum_{F\in\mathcal F_h^{\partial}}h^{-1}_F(\boldsymbol{\sigma}_h(\boldsymbol{u}_h),\boldsymbol{\varepsilon}_h(\boldsymbol{v}_h))_F,\\
b_h(\boldsymbol{u}_h,\boldsymbol{v}_h)&:=(\boldsymbol{\sigma}_h(\boldsymbol{u}_h), \boldsymbol{\varepsilon}_h(\boldsymbol{v}_h)).
\end{align*}
% with $\boldsymbol{\sigma}_h(\boldsymbol{u}_h)=2\mu\boldsymbol{\varepsilon}_h(\boldsymbol{u}_h)+\lambda(\div\boldsymbol{u}_h)\boldsymbol{I}$.
Notice that $a_h(\boldsymbol{u}_h,\boldsymbol{v}_h)=(\boldsymbol{\sigma}_h(\boldsymbol{u}_h), \boldsymbol{\varepsilon}_h(\boldsymbol{v}_h))_{1,h}$.
% Here the penalty constant $\eta$ is chosen to be sufficiently large to ensure the coercivity of the discrete bilinear form.

\begin{lemma}
We have the discrete Korn's inequality
\begin{equation}\label{discreteKorn}
\|\boldsymbol{v}\|_{1,h}\lesssim \| \boldsymbol{\varepsilon}_h(\boldsymbol{v})\|_{0} \quad \forall~\boldsymbol{v}\in  V_h,
\end{equation}
and the discrete $H^2$-Korn's inequality
\begin{equation}\label{discreteKornH2}
\|\boldsymbol{v}\|_{2,h}\lesssim \mathopen{\interleave} \boldsymbol{\varepsilon}_h(\boldsymbol{v})\mathclose{\interleave}_{1,h}  \quad \forall~\boldsymbol{v}\in  V_h.
\end{equation}
\end{lemma}
\begin{proof}
Thanks to (1.22) in \cite{MR2047078} and (3.8) in \cite{ChenHuHuang2018},
the discrete Korn's inequality \eqref{discreteKorn} follows from the weak continuity \eqref{weak-continuity-vh2}.

According to the inequality \eqref{eq:20260614}, it holds from the discrete Korn's inequality \eqref{discreteKorn} that
\begin{equation*}
\|\boldsymbol{v}\|_{2,h}\lesssim |\boldsymbol{\varepsilon}_h(\boldsymbol{v})|_{1,h}+\| \boldsymbol{\varepsilon}_h(\boldsymbol{v})\|_{0} \quad \forall~\boldsymbol{v}\in  V_h.
\end{equation*}
On the other side, employing the discrete Poincar\'e inequality in \cite[Remark 1.1]{MR1974504} and the weak continuity \eqref{weak-continuity-vh1}, we get
\begin{equation*}
\| \boldsymbol{\varepsilon}_h(\boldsymbol{v})\|_{0}\lesssim \mathopen{\interleave} \boldsymbol{\varepsilon}_h(\boldsymbol{v})\mathclose{\interleave}_{1,h} \quad\forall~\boldsymbol{v}\in V_h.
\end{equation*}
Combining the last two inequalities gives \eqref{discreteKornH2}.
\end{proof}

\begin{lemma}\label{normeq-lemma}
The following norm equivalence holds
\begin{align}
\label{normeq}\mathopen{\interleave}\boldsymbol{v}\mathclose{\interleave}^2_{\iota,\lambda,h}&\eqsim 2\mu\|\boldsymbol{v}\|^2_{1,h}+\lambda\|\div\boldsymbol{v}\|^2_{0}+\iota^2(2\mu\|\boldsymbol{v}\|^2_{2,h}+\lambda|\div\boldsymbol{v}|^2_{1,h})
\\
\notag&\quad\;+\iota^2\sum_{F\in\mathcal F_h}h^{-1}_F(2\mu\|[\![\boldsymbol{\varepsilon}_h(\boldsymbol{v})]\!]\|^2_{0,F}+\lambda\|[\![\div\boldsymbol{v}]\!]\|^2_{0,F}) \qquad\forall~\boldsymbol{v}\in V_h,
\end{align}
which implies that $\mathopen{\interleave}\cdot\mathclose{\interleave}_{\iota,\lambda,h}$ is a norm on space $V_h$.
\end{lemma}
\begin{proof}
By the weak continuity \eqref{weak-continuity-vh1}, we have for any $\boldsymbol{v}\in V_h$ that
\begin{equation*}
Q_{0,F}[\![\nabla_h\boldsymbol v]\!]=0
\qquad \forall \  F\in\mathring{\mathcal F}_h,
\end{equation*}
which combined with the discrete Korn's inequalities \eqref{discreteKorn}-\eqref{discreteKornH2} and the norm equivalence \eqref{eq:brokenH1seminormequiv} yields \eqref{normeq}.
\end{proof}

%By the norm equivalence \eqref{eq:brokenH1seminormequiv} and the discrete Korn's inequalities \eqref{discreteKorn}-\eqref{discreteKornH2},
%\begin{align*}
%\mathopen{\interleave}\boldsymbol{v}\mathclose{\interleave}^2_{\iota,\lambda,h}&\eqsim 2\mu|\boldsymbol{v}|^2_{1,h}+\lambda\|\div\boldsymbol{v}\|^2_{0}+\iota^2(2\mu|\boldsymbol{v}|^2_{2,h}+\lambda|\div\boldsymbol{v}|^2_{1,h})
%\\
%&\quad\;+\iota^2\sum_{F\in\mathcal F_h}h^{-1}_F(2\mu\|[\![\boldsymbol{\varepsilon}_h(\boldsymbol{v})]\!]\|^2_{0,F}+\lambda\|[\![\div\boldsymbol{v}]\!]\|^2_{0,F}) \qquad\forall~\boldsymbol{v}\in V_h.
%\end{align*}
%This implies $\mathopen{\interleave}\cdot\mathclose{\interleave}_{\iota,\lambda,h}$ is a norm on space $V_h$.

\begin{lemma}
\label{wellposed-lemma}
The nonconforming
finite element method \eqref{discrete-fem-weak} is well-posed, when $\eta$ is large enough.
\end{lemma}
\begin{proof}
There exists a constant $\eta_0>0$ depending only on the shape regularity of $\mathcal T_h$ such that for any fixed number $\eta\geq\eta_0$, it holds that (cf. \cite{EpshteynRiviere2007})
\begin{equation*}%\label{eq:ahelliptic}
2\mu\mathopen{\interleave} \boldsymbol{\varepsilon}_h(\boldsymbol{v})\mathclose{\interleave}_{1,h}^2 + \lambda\mathopen{\interleave}\div\boldsymbol{v}\mathclose{\interleave}^2_{1,h} \lesssim a_{h}(\boldsymbol{v}, \boldsymbol{v})  \quad \forall~\boldsymbol{v}\in  V_h.
\end{equation*}
% \begin{equation*}
% 2\mu|\boldsymbol{\varepsilon}_h(\boldsymbol{v})|_{1,h}^2 + \lambda|\div\boldsymbol{v}|^2_{1,h} + \sum_{F\in\mathcal F_h^{\partial}}h_F^{-1}(2\mu\|\boldsymbol{\varepsilon}_h(\boldsymbol{v})\|_{0,F}^2+\lambda\|\div\boldsymbol{v}\|_{0,F}^2)\lesssim a_{h}(\boldsymbol{v}, \boldsymbol{v}).
% \end{equation*}
% This together with the discrete $H^2$-Korn's inequality \eqref{discreteKornH2}
% gives
% \begin{equation*}
% 2\mu\|\boldsymbol{v}\|^2_{2,h}\lesssim  a_{h}(\boldsymbol{v}, \boldsymbol{v})  \quad \forall~\boldsymbol{v}\in  V_h.
% \end{equation*}
% By the discrete Korn's inequality \eqref{discreteKorn},
% \begin{equation}\label{eq:bhelliptic}
% 2\mu\|\boldsymbol{\varepsilon}_h(\boldsymbol{v})\|^2_{0}+\lambda\|\div\boldsymbol{v}\|^2_{0}\lesssim b_{h}(\boldsymbol{v}, \boldsymbol{v}) \quad \forall~\boldsymbol{v}\in  V_h.
% \end{equation}
%Then
Together with the coercivity of $b_h(\cdot,\cdot)$ and the continuity of the bilinear form, we obtain the equivalence
\begin{equation}\label{eq:abhelliptic}
\mathopen{\interleave}\boldsymbol{v}\mathclose{\interleave}^2_{\iota,\lambda,h}\eqsim \iota^2a_h(\boldsymbol{v}, \boldsymbol{v})+b_h(\boldsymbol{v}, \boldsymbol{v}) \quad \forall~\boldsymbol{v}\in  V_h.
\end{equation}
Hence, the well-posedness of the nonconforming finite element method \eqref{discrete-fem-weak} follows from \eqref{normeq} and the Lax-Milgram lemma \cite{Ciarlet1978}.
\end{proof}

\subsection{Error analysis}
%Let Lagrange element space
Denote the Lagrange finite element space of quadratic order by
\[ 
V^L_h:=\{\boldsymbol{v}\in H^1_0(\Omega;\mathbb{R}^d): \boldsymbol{v}|_T\in\mathbb{P}_2(T;\mathbb{R}^d) \quad \forall~T\in\mathcal{T}_h \}.  
\]
Define a connection operator $E_h: V_h\rightarrow V^L_h$ as follows: for any $\boldsymbol{v}\in V_h$,  $E_h\boldsymbol{v}$ is determined by
\[
\mathcal{N}(E_h\boldsymbol{v}):=\dfrac{1}{\#\mathcal{T}_{\mathcal{N}}}\sum_{T\in\mathcal T_\mathcal{N}}\mathcal{N}(\boldsymbol{v}|_T)
\]
%$$N(E_h\boldsymbol{v}):=N(\boldsymbol{v}|_T)$$
for each interior degree of freedom $\mathcal{N}$ of the space $ V^L_h$, where $\mathcal{T}_{\mathcal{N}}\subset\mathcal{T}_h$ denotes the set of simplices sharing the DoF $\mathcal{N}$. By the weak continuity \eqref{weak-continuity-vh2} of $ V_h$, we have (cf. \cite[Lemma 3.5]{Wang2001})
\begin{align}
\label{connection-Vh1}
|\boldsymbol{v}-E_h\boldsymbol{v}|_{m,T}\lesssim h_T^{s-m}|\boldsymbol{v}|_{s,\omega_T}\quad \forall \ \boldsymbol{v}\in V_h, 0\leq m\leq 1, 1\leq s\leq 2.
\end{align}
Here $|\boldsymbol v|_{s,\omega_T}^2:=\sum_{T'\in\mathcal T_h,\,T'\subset\omega_T}|\boldsymbol v|_{s,T'}^2$ denotes the broken Sobolev seminorm.
% Then we get
% \begin{align}
% %\label{connection-Vh3}|\boldsymbol{v}-E_h\boldsymbol{v}|_{1,h}&\lesssim
% %h^{1-r}|\boldsymbol{v}|^r_{1,h}\normmm{\boldsymbol{v}}^{1-r}_{2,h} \quad\forall \ \boldsymbol{v}\in V_h,\\
% \label{connection-Vh2}|\boldsymbol{v}-E_h\boldsymbol{v}|_{1,h}\lesssim h^{1-r}|\boldsymbol{v}|^r_{1,h}|\boldsymbol{v}|^{1-r}_{2,h} \quad\forall \ \boldsymbol{v}\in V_h,
% \end{align}
% with $0\leq r\leq 1$.

\begin{lemma}
\label{interpolation-I-lemma7}
Let $\boldsymbol{u}\in H_0^2(\Omega;\mathbb{R}^d)\cap H^3(\Omega;\mathbb{R}^d)$ be the solution of problem \eqref{SGE0}, and the solution $\boldsymbol{u}_0\in H^1_0(\Omega;\mathbb{R}^d)$ of problem \eqref{SGElinear} satisfies the regularity \eqref{elasregularity} with $2\leq s\leq 3$. We have for any $\boldsymbol{v}_h\in V_h$ that
\begin{align}
\label{Ih-error4}
&\quad\;\iota^2a_h(\boldsymbol{u}-I_h^V\boldsymbol{u},\boldsymbol{v}_h)+b_h(\boldsymbol{u}-I_h^V\boldsymbol{u},\boldsymbol{v}_h) \\
\notag
&\qquad\qquad\qquad\qquad\qquad\lesssim \iota^{-1/2}h\|\boldsymbol{f}\|_0(\|\boldsymbol{\varepsilon}_h(\boldsymbol{v}_h)\|_{0}+\iota\mathopen{\interleave} \boldsymbol{\varepsilon}_h(\boldsymbol{v}_h)\mathclose{\interleave}_{1,h}), \\
\label{Ih-error5}
&\quad\;\iota^2a_h(\boldsymbol{u}-I_h^V\boldsymbol{u},\boldsymbol{v}_h)+b_h(\boldsymbol{u}-I_h^V\boldsymbol{u},\boldsymbol{v}_h) \\
\notag
&\qquad\qquad\qquad\qquad\qquad\lesssim(\iota^{1/2}\|\boldsymbol{f}\|_0+h^{s-1}\|\boldsymbol{f}\|_{s-2})\|\boldsymbol{\varepsilon}_h(\boldsymbol{v}_h)\|_0.
\end{align}
% and
\end{lemma}
\begin{proof}
Employing the Cauchy-Schwarz inequality, the inverse inequality, (\ref{Ih-error1})-(\ref{Ih-error11}), (\ref{Jh-error1})-(\ref{Jh-error11}), and \eqref{interpolation-I-lemma2-pf},
\begin{align*}
a_h(\boldsymbol{u}-I_h^V\boldsymbol{u},\boldsymbol{v}_h)&\lesssim \mathopen{\interleave}\boldsymbol{\sigma}_h(\boldsymbol{u}-I_h^V\boldsymbol{u})\mathclose{\interleave}_{1,h}\mathopen{\interleave}\boldsymbol{\varepsilon}_h(\boldsymbol{v}_h)\mathclose{\interleave}_{1,h} \\
&\quad + \sum_{F\in\mathcal F_h^{\partial}}\|\partial_{n}(\boldsymbol{\sigma}_h(\boldsymbol{u}-I_h^V\boldsymbol{u}))\|_{0,F}\|\boldsymbol{\varepsilon}_h(\boldsymbol{v}_h)\|_{0,F} \\
&\lesssim h(|\boldsymbol{u}|_3+\lambda|\div\boldsymbol{u}|_2)\mathopen{\interleave}\boldsymbol{\varepsilon}_h(\boldsymbol{v}_h)\mathclose{\interleave}_{1,h}, \\
a_h(\boldsymbol{u}-I_h^V\boldsymbol{u},\boldsymbol{v}_h)&\lesssim \|\boldsymbol{\varepsilon}_h(\boldsymbol{v}_h)\|_{0}\Bigg(\sum_{T\in\mathcal{T}_h}h_T^{-2}\|\nabla\boldsymbol{\sigma}_h(\boldsymbol{u}-I_h^V\boldsymbol{u})\|_{0,T}^2\Bigg)^{1/2} \\
&\quad + \|\boldsymbol{\varepsilon}_h(\boldsymbol{v}_h)\|_{0}\Bigg(\sum_{F\in\mathcal{F}_h^{\partial}}\sum_{j=0}^1h_F^{2j-3}\|\partial_n^j\boldsymbol{\sigma}_h(\boldsymbol{u}-I_h^V\boldsymbol{u})\|_{0,F}^2\Bigg)^{1/2} \\
&\lesssim (|\boldsymbol{u}|_3+\lambda|\div\boldsymbol{u}|_2)\|\boldsymbol{\varepsilon}_h(\boldsymbol{v}_h)\|_{0},
\end{align*}
\begin{equation*}
b_h(\boldsymbol{u}-I_h^V\boldsymbol{u},\boldsymbol{v}_h)\leq\|\boldsymbol{\sigma}_h(\boldsymbol{u}-I_h^V\boldsymbol{u})\|_0\|\boldsymbol{\varepsilon}_h(\boldsymbol{v}_h)\|_0 \lesssim h(|\boldsymbol{u}|_2+\lambda|\div\boldsymbol{u}|_1)\|\boldsymbol{\varepsilon}_h(\boldsymbol{v}_h)\|_0,
\end{equation*}
and
\begin{align*}
& b_h(\boldsymbol{u}-I_h^V\boldsymbol{u},\boldsymbol{v}_h) \\
&\quad \lesssim (|\boldsymbol{u}-\boldsymbol{u}_0|_1 + \lambda\|\div(\boldsymbol{u}-\boldsymbol{u}_0)\|_0+h^{s-1}(|\boldsymbol{u}_0|_s+\lambda |\div\boldsymbol{u}_0|_{s-1}))\|\boldsymbol{\varepsilon}_h(\boldsymbol{v}_h)\|_0.
\end{align*}
Then estimates (\ref{Ih-error4})-(\ref{Ih-error5}) follow from (\ref{elasregularity})-(\ref{Regularity-divu}).
\end{proof}

\begin{lemma}
\label{consistency-lemma1}
Let $\boldsymbol{u}\in H_0^2(\Omega;\mathbb{R}^d)\cap H^3(\Omega;\mathbb{R}^d)$ be the solution of problem \eqref{SGE0}. We have for any $\boldsymbol{v}_h\in V_h$ that
\begin{align}
\label{consistency-error30}
a_h(\boldsymbol{u},\boldsymbol{v}_h)+(\Delta\boldsymbol{\sigma}(\boldsymbol{u}),\boldsymbol{\varepsilon}(E_h\boldsymbol{v}_h))&\lesssim \iota^{-3/2}h\|\boldsymbol{f}\|_0|\boldsymbol{v}_h|_{2,h}, \\
\label{consistency-error3}
a_h(\boldsymbol{u},\boldsymbol{v}_h)+(\Delta\boldsymbol{\sigma}(\boldsymbol{u}),\boldsymbol{\varepsilon}(E_h\boldsymbol{v}_h))&\lesssim \iota^{-3/2}\|\boldsymbol{f}\|_0|\boldsymbol{v}_h|_{1,h}.
\end{align}
\end{lemma}
\begin{proof}
Integrating by parts and using the weak continuity of $ V_h$, we have 
\begin{align*}
&\quad\; a_h(\boldsymbol{u},\boldsymbol{v}_h)+(\Delta\boldsymbol{\sigma}(\boldsymbol{u}),\boldsymbol{\varepsilon}(E_h\boldsymbol{v}_h))\\
&=(\Delta\boldsymbol{\sigma}(\boldsymbol{u}),\boldsymbol{\varepsilon}_h(E_h\boldsymbol{v}_h-\boldsymbol{v}_h))+\sum_{F\in\mathring{\mathcal{F}}_h}(\partial_{n}\boldsymbol{\sigma}(\boldsymbol{u}),[\![\boldsymbol{\varepsilon}_h(\boldsymbol{v}_h)]\!])_F\\
&=(\Delta\boldsymbol{\sigma}(\boldsymbol{u}),\boldsymbol{\varepsilon}_h(E_h\boldsymbol{v}_h-\boldsymbol{v}_h))\\
&\quad+\sum_{F\in\mathring{\mathcal{F}}_h}(\partial_{n}\boldsymbol{\sigma}(\boldsymbol{u})-Q_{0,F}\partial_{n}\boldsymbol{\sigma}(\boldsymbol{u}),[\![\boldsymbol{\varepsilon}_h(\boldsymbol{v}_h)]\!]-Q_{0,F}[\![\boldsymbol{\varepsilon}_h(\boldsymbol{v}_h)]\!])_F.
\end{align*}
Applying the error estimate of $Q_{0,F}$, (\ref{connection-Vh1}) and the inverse inequality, we deduce
\begin{align*}
& a_h(\boldsymbol{u},\boldsymbol{v}_h) + (\Delta\boldsymbol{\sigma}(\boldsymbol{u}),\boldsymbol{\varepsilon}(E_h\boldsymbol{v}_h)) \lesssim h|\boldsymbol{\sigma}(\boldsymbol{u})|_2|\boldsymbol{v}_h|_{2,h}, \\
& a_h(\boldsymbol{u},\boldsymbol{v}_h) + (\Delta\boldsymbol{\sigma}(\boldsymbol{u}),\boldsymbol{\varepsilon}(E_h\boldsymbol{v}_h)) \lesssim |\boldsymbol{\sigma}(\boldsymbol{u})|_2|\boldsymbol{v}_h|_{1,h}.
\end{align*}
Hence, the estimates (\ref{consistency-error30})-(\ref{consistency-error3}) follow from the regularity (\ref{Regularity-u})-(\ref{Regularity-divu}).
%Hence, the estimates (\ref{consistency-error30})-(\ref{consistency-error3}) follow from the regularity (\ref{Regularity-u})-(\ref{Regularity-divu}) and the discrete Korn's inequalities \eqref{discreteKorn}-\eqref{discreteKornH2} immediately.
\end{proof}

\begin{lemma}
\label{consistency-lemma2}
Let $\boldsymbol{u}\in H_0^2(\Omega;\mathbb{R}^d)\cap H^3(\Omega;\mathbb{R}^d)$ be the solution of problem \eqref{SGE0}, and the solution $\boldsymbol{u}_0\in H^1_0(\Omega;\mathbb{R}^d)$ of problem \eqref{SGElinear} satisfies the regularity \eqref{elasregularity} with $2\leq s\leq 3$. We have for any $\boldsymbol{v}_h\in V_h$ that
\begin{align}
\label{main-proof-uh-20}
&(\boldsymbol{f},\boldsymbol{v}_h)-\iota^2 a_h(\boldsymbol{u},\boldsymbol{v}_h)-b_h(\boldsymbol{u},\boldsymbol{v}_h) \!\lesssim h\|\boldsymbol{f}\|_0(\|\boldsymbol{\varepsilon}_h(\boldsymbol{v}_h)\|_{0}+\iota^{1/2}\mathopen{\interleave} \boldsymbol{\varepsilon}_h(\boldsymbol{v}_h)\mathclose{\interleave}_{1,h}), \\
\label{main-proof-uh-2}
&(\boldsymbol{f},\boldsymbol{v}_h)-\iota^2 a_h(\boldsymbol{u},\boldsymbol{v}_h)-b_h(\boldsymbol{u},\boldsymbol{v}_h) \lesssim (\iota^{1/2}\|\boldsymbol{f}\|_0+h^{s-1}\|\boldsymbol{f}\|_{s-2})\|\boldsymbol{\varepsilon}_h(\boldsymbol{v}_h)\|_{0}.
\end{align}
\end{lemma}
\begin{proof}
Let $\boldsymbol{w}_h=\boldsymbol{v}_h-E_h\boldsymbol{v}_h$ for simplicity. 
Apply integration by parts to problem (\ref{SGElinear}) to get
\begin{equation}\label{eq:202408081}
(\boldsymbol{f}, \boldsymbol{w}_h)-b_h(\boldsymbol{u},\boldsymbol{w}_h)=(\boldsymbol{\sigma}(\boldsymbol{u}_0-\boldsymbol{u}),\boldsymbol{\varepsilon}_h(\boldsymbol{w}_h))-\sum_{F\in\mathcal{F}_h}(\boldsymbol{\sigma}(\boldsymbol{u}_0)\boldsymbol{n},[\![\boldsymbol{v}_h]\!])_F.
\end{equation}
Adopting (\ref{connection-Vh1}), it holds
\begin{equation}\label{eq:202408082}
(\boldsymbol{\sigma}(\boldsymbol{u}_0-\boldsymbol{u}),\boldsymbol{\varepsilon}_h(\boldsymbol{w}_h))
\lesssim \|\boldsymbol{\sigma}(\boldsymbol{u}_0-\boldsymbol{u})\|_0\min\{|\boldsymbol{v}_h|_{1,h}, h|\boldsymbol{v}_h|_{2,h}\}.
%\min\{\|\boldsymbol{\varepsilon}_h(\boldsymbol{v}_h)\|_{0}, h\mathopen{\interleave} \boldsymbol{\varepsilon}_h(\boldsymbol{v})\mathclose{\interleave}_{1,h}\}.
\end{equation}
Thanks to the weak continuity \eqref{weak-continuity-vh2} of $V_{h}$, we obtain
\[-\sum_{F\in\mathcal{F}_h}(\boldsymbol{\sigma}(\boldsymbol{u}_0)\boldsymbol{n},[\![\boldsymbol{v}_h]\!])_F=-\sum_{F\in\mathcal{F}_h}(\boldsymbol{\sigma}(\boldsymbol{u}_0)\boldsymbol{n}-Q_{1,F}\boldsymbol{\sigma}(\boldsymbol{u}_0)\boldsymbol{n},[\![\boldsymbol{v}_h]\!]-Q_{1,F}[\![\boldsymbol{v}_h]\!])_F.\]
From the error estimate of $Q_{1,F}$, the trace inequality and (\ref{elasregularity}), we acquire
\[
-\sum_{F\in\mathcal{F}_h}(\boldsymbol{\sigma}(\boldsymbol{u}_0)\boldsymbol{n},[\![\boldsymbol{v}_h]\!])_F\lesssim h^{s-1}\|\boldsymbol{f}\|_{s-2}|\boldsymbol{v}_h|_{1,h}.
\]
Then it holds from \eqref{eq:202408081}-\eqref{eq:202408082} and the regularity \eqref{Regularity-u}-\eqref{Regularity-divu} that
\begin{align}
\label{consistency-error4}
(\boldsymbol{f},\boldsymbol{w}_h)-b_h(\boldsymbol{u},\boldsymbol{w}_h)&\lesssim \iota^{1/2}\|\boldsymbol{f}\|_0\min\{|\boldsymbol{v}_h|_{1,h}, h|\boldsymbol{v}_h|_{2,h}\} \\
&\quad+h^{s-1}\|\boldsymbol{f}\|_{s-2}|\boldsymbol{v}_h|_{1,h}.\notag
\end{align}
%\begin{align}
%	\label{consistency-error4}
%	(\boldsymbol{f},\boldsymbol{w}_h)-b_h(\boldsymbol{u},\boldsymbol{w}_h)&\lesssim \iota^{1/2}h\|\boldsymbol{f}\|_0|\boldsymbol{v}_h|_{2,h}+h^{s-1}\|\boldsymbol{f}\|_{s-2}|\boldsymbol{v}_h|_{1,h},\\
%	\label{consistency-error5}
%	(\boldsymbol{f},\boldsymbol{w}_h)-b_h(\boldsymbol{u},\boldsymbol{w}_h)&\lesssim (\iota^{1/2}\|\boldsymbol{f}\|_0
%	+h^{s-1}\|\boldsymbol{f}\|_{s-2})|\boldsymbol{v}_h|_{1,h}.
%\end{align}
On the other hand, apply integration by parts to problem \eqref{SGE0} to get
\begin{align*}
(\boldsymbol{f},E_h\boldsymbol{v}_h)=-\iota^2(\Delta\boldsymbol{\sigma}(\boldsymbol{u}),\boldsymbol{\varepsilon}(E_h\boldsymbol{v}_h))+(\boldsymbol{\sigma}(\boldsymbol{u}),\boldsymbol{\varepsilon}(E_h\boldsymbol{v}_h)).
\end{align*}
So we have
\begin{align*}
(\boldsymbol{f},\boldsymbol{v}_h)-\iota^2a_h(\boldsymbol{u},\boldsymbol{v}_h)-b_h(\boldsymbol{u},\boldsymbol{v}_h)
&=(\boldsymbol{f},\boldsymbol{w}_h)-b_h(\boldsymbol{u},\boldsymbol{w}_h) \\
&\quad\;-\iota^2a_h(\boldsymbol{u},\boldsymbol{v}_h)-\iota^2(\Delta\boldsymbol{\sigma}(\boldsymbol{u}),\boldsymbol{\varepsilon}(E_h\boldsymbol{v}_h)).
\end{align*}
Combining (\ref{consistency-error30})-(\ref{consistency-error3}) and (\ref{consistency-error4}) gives
\begin{align*}
(\boldsymbol{f},\boldsymbol{v}_h)-\iota^2a_h(\boldsymbol{u},\boldsymbol{v}_h)-b_h(\boldsymbol{u},\boldsymbol{v}_h)&\lesssim \iota^{1/2}h\|\boldsymbol{f}\|_0|\boldsymbol{v}_h|_{2,h}+h\|\boldsymbol{f}\|_{0}|\boldsymbol{v}_h|_{1,h},\\
(\boldsymbol{f},\boldsymbol{v}_h)-\iota^2a_h(\boldsymbol{u},\boldsymbol{v}_h)-b_h(\boldsymbol{u},\boldsymbol{v}_h)&\lesssim (\iota^{1/2}\|\boldsymbol{f}\|_0
+h^{s-1}\|\boldsymbol{f}\|_{s-2})|\boldsymbol{v}_h|_{1,h}.
\end{align*}
The estimates \eqref{main-proof-uh-20}-\eqref{main-proof-uh-2} follow from the last two inequalities and the discrete Korn's inequalities~\eqref{discreteKorn}-\eqref{discreteKornH2}.
\end{proof}

\begin{theorem}
\label{main-theorem-uh-1}
Let $\boldsymbol{u}\in H_0^2(\Omega;\mathbb{R}^d)\cap H^3(\Omega;\mathbb{R}^d)$ be the solution of problem \eqref{SGE0}, and the solution $\boldsymbol{u}_0\in H^1_0(\Omega;\mathbb{R}^d)$ of problem \eqref{SGElinear} satisfies the regularity \eqref{elasregularity} with $2\leq s\leq 3$. Then
we have
\begin{align}
\label{main-result-uh-1}
\mathopen{\interleave}\boldsymbol{u}-\boldsymbol{u}_h\mathclose{\interleave}_{\iota,\lambda,h} &\lesssim \min\{\iota^{-1/2}h\|\boldsymbol{f}\|_0, \iota^{1/2}\|\boldsymbol{f}\|_0+h^{s-1}\|\boldsymbol{f}\|_{s-2}\}, \\
\label{main-result-uh-2}
\|\boldsymbol{u}_0-\boldsymbol{u}_{h}\|_{\iota,\lambda,h}&\lesssim \iota^{1/2}\|\boldsymbol{f}\|_0+h^{s-1}\|\boldsymbol{f}\|_{s-2}.
\end{align}
\end{theorem}
\begin{proof}
Let $\boldsymbol{v}_h=\boldsymbol{u}_h-I_h^V\boldsymbol{u}$ for simplicity. By \eqref{eq:abhelliptic} and (\ref{discrete-fem-weak}), it holds
\begin{align*}
\mathopen{\interleave}\boldsymbol{u}_h-I_h^V\boldsymbol{u}\mathclose{\interleave}^2_{\iota,\lambda,h} &\lesssim\iota^2a_h(\boldsymbol{u}_h-I_h^V\boldsymbol{u},\boldsymbol{v}_h)+b_h(\boldsymbol{u}_h-I_h^V\boldsymbol{u},\boldsymbol{v}_h) \\
& = (\boldsymbol{f},\boldsymbol{v}_h)- \iota^2a_h(I_h^V\boldsymbol{u},\boldsymbol{v}_h)-b_h(I_h^V\boldsymbol{u},\boldsymbol{v}_h)\\
&=(\boldsymbol{f},\boldsymbol{v}_h)-\iota^2a_h(\boldsymbol{u},\boldsymbol{v}_h)-b_h(\boldsymbol{u},\boldsymbol{v}_h) \\
&\quad\;+\iota^2a_h(\boldsymbol{u}-I_h^V\boldsymbol{u},\boldsymbol{v}_h)+b_h(\boldsymbol{u}-I_h^V\boldsymbol{u},\boldsymbol{v}_h).
\end{align*}
By \eqref{Ih-error4} and \eqref{main-proof-uh-20}, we obtain
\begin{align*}
\mathopen{\interleave}\boldsymbol{u}_h-I_h^V\boldsymbol{u}\mathclose{\interleave}^2_{\iota,\lambda,h}
&\lesssim
\iota^{-1/2}h\|\boldsymbol f\|_0
(\|\boldsymbol\varepsilon_h(\boldsymbol v_h)\|_0
+
\iota\mathopen{\interleave}\boldsymbol\varepsilon_h(\boldsymbol v_h)\mathclose{\interleave}_{1,h})  \\
&\quad
+
h\|\boldsymbol f\|_0
(
\|\boldsymbol\varepsilon_h(\boldsymbol v_h)\|_0
+
\iota^{1/2}\mathopen{\interleave}\boldsymbol\varepsilon_h(\boldsymbol v_h)\mathclose{\interleave}_{1,h})\\
&\lesssim\iota^{-1/2}h\|\boldsymbol f\|_0(\|\boldsymbol\varepsilon_h(\boldsymbol v_h)\|_0
+
\iota\mathopen{\interleave}\boldsymbol\varepsilon_h(\boldsymbol v_h)\mathclose{\interleave}_{1,h}).
\end{align*}
On the other hand, using \eqref{Ih-error5} and \eqref{main-proof-uh-2} gives
\begin{equation*}
\mathopen{\interleave}\boldsymbol{u}_h-I_h^V\boldsymbol{u}\mathclose{\interleave}^2_{\iota,\lambda,h}
\lesssim
\bigl(
\iota^{1/2}\|\boldsymbol f\|_0
+
h^{s-1}\|\boldsymbol f\|_{s-2}
\bigr)
\|\boldsymbol\varepsilon_h(\boldsymbol v_h)\|_0 .
\end{equation*}
Combining the last two estimates yields
\begin{equation*}
\mathopen{\interleave}\boldsymbol{u}_h-I_h^V\boldsymbol{u}\mathclose{\interleave}_{\iota,\lambda,h}
\lesssim
\min\{
\iota^{-1/2}h\|\boldsymbol f\|_0,\,
\iota^{1/2}\|\boldsymbol f\|_0+h^{s-1}\|\boldsymbol f\|_{s-2}
\},
\end{equation*}
%Then we get from (\ref{Ih-error4})-(\ref{Ih-error5}) and (\ref{main-proof-uh-20})-(\ref{main-proof-uh-2}) that
%\begin{equation*}
%\mathopen{\interleave}\boldsymbol{u}_h-I_h^V\boldsymbol{u}\mathclose{\interleave}_{\iota,\lambda,h} \lesssim \min\{\iota^{-1/2}h\|\boldsymbol{f}\|_0, \iota^{1/2}\|\boldsymbol{f}\|_0+h^{s-1}\|\boldsymbol{f}\|_{s-2}\}.
%\end{equation*}
which together with (\ref{Regularity-u})-(\ref{Regularity-divu}) and (\ref{Ih-error3})-(\ref{Ih-error2})
leads to \eqref{main-result-uh-1}.

Finally, it follows from (\ref{elasregularity})-(\ref{Regularity-divu}) that
\begin{equation*}
\|\boldsymbol{u}_0-\boldsymbol{u}\|_{\iota,\lambda,h}\lesssim \iota^{1/2}\|\boldsymbol{f}\|_0,
\end{equation*}
which combined with (\ref{main-result-uh-1}) yields (\ref{main-result-uh-2}).
%Finally, we conclude (\ref{main-result-uh-2}) from (\ref{main-result-uh-1}) and (\ref{elasregularity})-(\ref{Regularity-divu}).
\end{proof}

% We remark that if $\iota\ll h$, then (\ref{main-result-uh-2}) can be reformulated as
% \begin{align}
% \normmm{\boldsymbol{u}_0-\boldsymbol{u}_h}_{\iota,\lambda,h}\lesssim\iota^{1/2}\|\boldsymbol{f}\|_0+h^{s-1}\|\boldsymbol{f}\|_{s-2}.
% \end{align}
\begin{remark}\label{remark-V-h0}\rm
All boundary conditions in the SGE model \eqref{SGE0} can be imposed strongly. Define
\begin{equation*}
V_{h0}=\{\boldsymbol{v}_{h}\in  V_{h}: \textrm{DoFs \eqref{dof4}-\eqref{dof5} on boundary vanish}\}.
\end{equation*}
We then propose the following finite element method for problem \eqref{weak1}:
find $\boldsymbol{u}_{h0}\in V_{h0}$ such that
\begin{align}
\label{discrete-fem-strong}
\iota^2(\nabla_h\boldsymbol{\sigma}_h(\boldsymbol{u}_{h0}),\nabla_h\boldsymbol{\varepsilon}_h(\boldsymbol{v}_h))+(\boldsymbol{\sigma}_h(\boldsymbol{u}_{h0}),\boldsymbol{\varepsilon}_h(\boldsymbol{v}_h))=(\boldsymbol{f},\boldsymbol{v}_h) \quad \forall~\boldsymbol{v}_h\in V_{h0}.
\end{align}
% Applying the similar argument in \cite{MR4650917}, we can get the following error estimates
We can derive the following error estimates
\begin{align}
\notag% \label{main-result-u0h-2}
\|\boldsymbol{u}-\boldsymbol{u}_{h0}\|_{\iota,\lambda,h}&\lesssim \min\{\iota^{-1/2}h\|\boldsymbol{f}\|_0, h^{1/2}\|\boldsymbol{f}\|_0\} , \\
\label{main-result-u0h-1}
\|\boldsymbol{u}_0-\boldsymbol{u}_{h0}\|_{\iota,\lambda,h}&\lesssim(\iota^{1/2}+h^{1/2})\|\boldsymbol{f}\|_0.
\end{align}
For a proof outline, define $I_{h0}^V\boldsymbol u\in V_{h0}$ by setting the boundary DoFs \eqref{dof4}--\eqref{dof5} of $I_h^V\boldsymbol u$ to zero. Scaling bounds this correction by boundary traces of the BDM interpolation error. The multiplicative trace inequality and \eqref{Regularity-u}--\eqref{Regularity-divu}, with $\boldsymbol u=(\boldsymbol u-\boldsymbol u_0)+\boldsymbol u_0$ for $\iota\leq h$, give
\[
\|\boldsymbol u-I_{h0}^V\boldsymbol u\|_{\iota,\lambda,h}\lesssim\min\{\iota^{-1/2}h,h^{1/2}\}\|\boldsymbol f\|_0.
\]
The consistency argument in Lemmas~\ref{consistency-lemma1}--\ref{consistency-lemma2} gives the same bounds, since the mean gradient jumps now vanish on boundary faces as well. Strang's lemma and the triangle inequality then yield the two estimates.
That is $\|\boldsymbol{u}-\boldsymbol{u}_{h0}\|_{\iota,\lambda,h}=O(h^{1/2})$ for small $\iota$.
% \begin{align}
% \label{main-result-u0h-1}
% \|\boldsymbol{u}_0-\boldsymbol{u}_{h0}\|_{\iota,\lambda,h}+\lambda\|\div(\boldsymbol{u}_0-\boldsymbol{u}_{h0})\|_0+\lambda\iota|\div(\boldsymbol{u}_0-\boldsymbol{u}_{h0})|_{1,h}\lesssim(\iota^{1/2}+h^{1/2})\|\boldsymbol{f}\|_0,
% \end{align}
% and
% \begin{align}\begin{split}
% \label{main-result-u0h-2}
% &\|\boldsymbol{u}-\boldsymbol{u}_{h0}\|_{\iota,\lambda,h}+\lambda\|\div(\boldsymbol{u}-\boldsymbol{u}_{h0})\|_0+\lambda\iota|\div(\boldsymbol{u}-\boldsymbol{u}_{h0})|_{1,h}\\
% \lesssim& (h+\iota^{r-1/2}h^{1-r})\|\boldsymbol{f}\|_0+(\iota h+h^2)(|\boldsymbol{u}|_3+\lambda|\div\boldsymbol{u}|_2),
% \end{split}
% \end{align}
% with $0\leq r\leq 1$.
\end{remark}

\begin{remark}\rm
\label{rem:mixed-bc}
The discrete formulation \eqref{discrete-fem-weak} can be adapted to the homogeneous mixed boundary conditions in Section~\ref{subsec:mixed-bc}. The displacement conditions on $\Gamma_{\rm cl}\cup\Gamma_{\rm ss}$ are imposed through the finite element space by prescribing the corresponding boundary DoFs. The normal derivative condition on $\Gamma_{\rm cl}$ is imposed weakly by Nitsche's technique. For a mesh-compatible boundary partition, the boundary sums in $a_h$ are restricted to faces in $\Gamma_{\rm cl}$, while the natural boundary functionals vanish. Nonhomogeneous natural data require sufficient regularity for elementwise boundary integration, or a discrete extension of the functional in $(V_0^{\tr})'$. A numerical example with homogeneous mixed boundary conditions is given in Example~\ref{example3}.
\end{remark}

%%%%%%%%%%%%%%%%%%%%%%%%%%%%%%%%%%%%%%%%
%% Numerical results
%%%%%%%%%%%%%%%%%%%%%%%%%%%%%%%%%%%%%%%%
\section{Numerical results}\label{sec5}
In this section, we will numerically examine the performance of the finite element methods (\ref{discrete-fem-weak}) and (\ref{discrete-fem-strong}).
Let $\Omega$ be the unit square $(0, 1)^2$, $\mu = 1$ and $\eta=100$.
All the numerical tests are performed on uniform triangulations.

\begin{example}\label{example1}
\normalfont
We first test the discrete method (\ref{discrete-fem-weak}) with the exact solution
\[
\boldsymbol{u}=
\left(
\begin{matrix}
\sin^3(\pi x)\sin(2\pi y)\sin(\pi y)\\
-\sin^3(\pi y)\sin(2\pi x)\sin(\pi x)
\end{matrix}
\right),
\]
which is divergence-free without boundary layers.
The right-hand side $\boldsymbol{f}$ is computed from \eqref{SGE0}, and is independent of the Lam\'{e} constant $\lambda$. 

Numerical errors $\mathopen{\interleave}\boldsymbol{u}-\boldsymbol{u}_{h}\mathclose{\interleave}_{\iota,\lambda,h}$ for $\lambda = 1$ and $\lambda = 10^6$ are presented in Table~\ref{table1} and Table~\ref{table2}, respectively.
The mesh size $h$ decreases from $2^{-3}$ to $2^{-7}$, and the size parameter $\iota$ from $1$ to $10^{-8}$.
We observe from Table \ref{table1} and Table \ref{table2} that $\mathopen{\interleave}\boldsymbol{u}-\boldsymbol{u}_{h}\mathclose{\interleave}_{\iota,\lambda,h}\eqsim O(h)$ for $\iota = 1, 10^{-2}$, and $\mathopen{\interleave}\boldsymbol{u}-\boldsymbol{u}_{h}\mathclose{\interleave}_{\iota,\lambda,h}\eqsim O(h^2)$ for $\iota = 10^{-4}, 10^{-6}, 10^{-8}$, 
% It can be seen that the method we proposed guarantees the optimal convergence rate when $\iota$ tends to zero, 
which are optimal and consistent with the theoretical result (\ref{main-result-uh-1}). 
% Moerover, by comparing the convergence rate results of Table \ref{table1} and Table \ref{table2}, we can draw conclusions about the robustness of nonconforming FEM (\ref{discrete-fem-weak}) with respect to $\lambda$.

% We present the error of $\mathopen{\interleave}\boldsymbol{u}-\boldsymbol{u}_{h}\mathclose{\interleave}_{\iota,\lambda,h}$ in Table \ref{table1} for various $h$ and $\iota$ while taking $\lambda = 1$. Under the same parameter settings, we test the
% robustness with respect to $\lambda$ by taking $\lambda = 10^6$ in Table \ref{table2}.

% In this example, we verify the optimal convergence results of the discrete method (\ref{discrete-fem-weak})
% for a sufficiently smooth solution without the boundary layer. From Table \ref{table1} and Table \ref{table2}, we observe that the convergence rates of  ${\mathopen{\interleave}\boldsymbol{u}-\boldsymbol{u}_{h}\mathopen{\interleave}}_{\iota,\lambda,h}$ for $\iota = 1, 10^{-2}$ are close to $O(h)$, while ones for $\iota = 10^{-6}, 10^{-8}$ would be close to $O(h^2)$. It can be seen that the method we proposed guarantees the optimal convergence rate when $\iota$ tends to zero, which is consistent with the theoretical result (\ref{main-result-uh-1}). Moerover, by comparing the convergence rate results of Table \ref{table1} and Table \ref{table2}, we can draw conclusions about the robustness of nonconforming FEM (\ref{discrete-fem-weak}) with respect to $\lambda$.
\begin{table}
\centering
\caption{Error $\mathopen{\interleave}\boldsymbol{u}-\boldsymbol{u}_{h}\mathclose{\interleave}_{\iota,\lambda,h}$ of the discrete method (\ref{discrete-fem-weak}) for Example \ref{example1} with $\lambda = 1$.  }
\vspace{-1.0em}
\label{table1}
\begin{tabular}{cccccc}
\toprule
$\iota\backslash h$ & 1/8 & 1/16 & 1/32 & 1/64 & 1/128\\
\midrule
1 & 1.242e+01 & 6.132e+00 & 3.056e+00 & 1.533e+00 & 7.740e-01\\
rate &  & 1.02 & 1.01 & 1.00 & 0.99\\
$10^{-2}$ & 2.943e-01 & 9.317e-02 & 3.518e-02 & 1.593e-02 & 7.816e-03\\
rate &  & 1.66 & 1.41 & 1.14 & 1.03\\
$10^{-4}$ & 2.424e-01 & 6.361e-02 & 1.606e-02 & 4.030e-03 & 1.014e-03\\
rate &  & 1.93 & 1.99 & 1.99 & 1.99\\
$10^{-6}$ & 2.421e-01 & 6.347e-02 & 1.601e-02 & 4.009e-03 & 1.002e-03\\
rate &  & 1.93 & 1.99 & 2.00 & 2.00\\
$10^{-8}$ & 2.421e-01 & 6.347e-02 & 1.601e-02 & 4.009e-03 & 1.002e-03\\
rate &  & 1.93 & 1.99 & 2.00 & 2.00\\
\bottomrule
\end{tabular}
\vspace{1.0em}
\end{table}

\begin{table}
\centering
\caption{Error $\mathopen{\interleave}\boldsymbol{u}-\boldsymbol{u}_{h}\mathclose{\interleave}_{\iota,\lambda,h}$ of the discrete method (\ref{discrete-fem-weak}) for Example \ref{example1} with $\lambda = 10^6$.}
\vspace{-1.0em}
\label{table2}
\begin{tabular}{cccccc}
\toprule
$\iota\backslash h$ & 1/8 & 1/16 & 1/32 & 1/64 & 1/128\\
\midrule
1 & 1.816e+01 & 1.341e+01 & 8.336e+00 & 4.539e+00 & 2.335e+00\\
rate &  & 0.44 & 0.69 & 0.88 & 0.96\\
$10^{-2}$ & 4.005e-01 & 1.936e-01 & 9.523e-02 & 4.719e-02 & 2.354e-02\\
rate &  & 1.05 & 1.02 & 1.01 & 1.00\\
$10^{-4}$ & 3.138e-01 & 1.051e-01 & 3.030e-02 & 7.979e-03 & 2.050e-03\\
rate &  & 1.58 & 1.79 & 1.93 & 1.96\\
$10^{-6}$ & 3.133e-01 & 1.049e-01 & 3.018e-02 & 7.905e-03 & 2.002e-03\\
rate &  & 1.58 & 1.80 & 1.93 & 1.98\\
$10^{-8}$ & 3.133e-01 & 1.049e-01 & 3.018e-02 & 7.905e-03 & 2.002e-03\\
rate &  & 1.58 & 1.80 & 1.93 & 1.98\\
\bottomrule
\end{tabular}
\vspace{1.0em}
\end{table}
\end{example}

\begin{example}\label{example2}
\normalfont
Next, we investigate the convergence of the discrete schemes \eqref{discrete-fem-weak} and \eqref{discrete-fem-strong} in the presence of boundary layers.
The exact solution of the linear
elasticity problem (\ref{SGElinear}) is set to be the divergence-free function
\[
\boldsymbol{u}_0=
\left(
\begin{matrix}
(e^{\cos(2\pi x)}-e)\sin(2\pi y)e^{\cos(2\pi y)}\\
-(e^{\cos(2\pi y)}-e)\sin(2\pi x)e^{\cos(2\pi x)}
\end{matrix}
\right).
\]
We take the right-hand side $\boldsymbol{f}$ computed from (\ref{SGElinear}) as the right-hand side of problem \eqref{SGE0}. Notice that $\boldsymbol{f}$ is independent of both the size parameter $\iota$ and the Lam\'{e} constant $\lambda$. 
%Since $\partial_{n}\boldsymbol{u}_0|_{\partial\Omega}= 0$, the solution $\boldsymbol{u}$ for problem \eqref{SGE0} has a strong boundary layer when $\iota$ is sufficient small.
Although the explicit expression of the solution $\boldsymbol{u}$ for problem \eqref{SGE0} is unknown, the solution $\boldsymbol{u}$ possesses strong boundary layers when $\iota$ is very small.
%The explicit expression of the solution $\boldsymbol{u}$ for problem \eqref{SGE0} is unknown, which possesses strong boundary layers when $\iota$ is very small. 
Take $\iota = 10^{-6}, 10^{-8}$.

Numerical error $\|\boldsymbol{u}_0-\boldsymbol{u}_{h0}\|_{\iota,\lambda,h}$  of the discrete method (\ref{discrete-fem-strong}) for $\lambda = 1$ and $\lambda = 10^6$ is presented in Table~\ref{table34}.
From Table~\ref{table34}, $\|\boldsymbol{u}_0-\boldsymbol{u}_{h0}\|_{\iota,\lambda,h}\eqsim O(h^{0.5})$, which is in coincidence with the estimate (\ref{main-result-u0h-1}).
So the estimate (\ref{main-result-u0h-1}) is robust with respect to the size parameter $\iota$ and the Lam\'{e} constant $\lambda$, which is suboptimal but sharp.

Finally, we list the numerical error $\|\boldsymbol{u}_0-\boldsymbol{u}_{h}\|_{\iota,\lambda,h}$  of the discrete method (\ref{discrete-fem-weak}) for $\lambda = 1$ and $\lambda = 10^6$ in Table~\ref{table56}. We can see from Table~\ref{table56} that $\|\boldsymbol{u}_0-\boldsymbol{u}_{h}\|_{\iota,\lambda,h}\eqsim O(h^{2})$, which agrees with the estimate \eqref{main-result-uh-2}.
Therefore, the discrete method (\ref{discrete-fem-weak}) is not only robust with respect to the size parameter $\iota$ and the Lam\'{e} constant $\lambda$, but also optimal.

\begin{table}
\centering
\caption{Error $\|\boldsymbol{u}_0-\boldsymbol{u}_{h0}\|_{\iota,\lambda,h}$ of the discrete method (\ref{discrete-fem-strong}) for Example \ref{example2}.}
% \caption{The performance of the discrete method (\ref{discrete-fem-strong}) for Example \ref{example2} with $\lambda = 1$.  }
\vspace{-1.0em}
\label{table34}
\begin{tabular}{ccccccc}
\toprule
$\lambda$ & $\iota\backslash h$ & 1/8 & 1/16 & 1/32 & 1/64 & 1/128\\
\midrule
\multirow{4}{*}{$1$} & $10^{-6}$ & 4.678e+00 & 2.956e+00 & 2.027e+00 & 1.423e+00 & 1.004e+00\\
& rate &  & 0.66 & 0.54 & 0.51 & 0.50\\
& $10^{-8}$& 4.678e+00 & 2.956e+00 & 2.027e+00 & 1.423e+00 & 1.004e+00\\
& rate &  & 0.66 & 0.54 & 0.51 & 0.50\\
\midrule
\multirow{4}{*}{$10^{6}$} & $10^{-6}$ & 5.291e+00 & 3.296e+00 & 2.220e+00 & 1.546e+00 & 1.089e+00\\
& rate &  & 0.68 & 0.57 & 0.52 & 0.51\\
& $10^{-8}$ & 5.291e+00 & 3.296e+00 & 2.220e+00 & 1.546e+00 & 1.089e+00\\
& rate &  & 0.68 & 0.57 & 0.52 & 0.51\\
\bottomrule
\end{tabular}
\vspace{1.0em}
\end{table}
\begin{table}
\centering
\caption{Error $\|\boldsymbol{u}_0-\boldsymbol{u}_h\|_{\iota,\lambda,h}$ of the discrete method (\ref{discrete-fem-weak}) for Example \ref{example2}.}
% \caption{The performance of the discrete method (\ref{discrete-fem-strong}) for Example \ref{example2} with $\lambda = 1$.  }
\vspace{-1.0em}
\label{table56}
\begin{tabular}{ccccccc}
\toprule
$\lambda$ & $\iota\backslash h$ & 1/8 & 1/16 & 1/32 & 1/64 & 1/128\\
\midrule
\multirow{4}{*}{$1$}& $10^{-6}$ & 1.523e+00 & 4.195e-01 & 1.071e-01 & 2.690e-02 & 6.731e-03\\
& rate &  & 1.86 & 1.97 & 1.99 & 2.00\\
& $10^{-8}$ & 1.523e+00 & 4.195e-01 & 1.071e-01 & 2.689e-02 & 6.730e-03\\
& rate &  & 1.86 & 1.97 & 1.99 & 2.00\\
\midrule
\multirow{4}{*}{$10^{6}$}& $10^{-6}$ & 2.225e+00 & 6.630e-01 & 1.840e-01 & 4.816e-02 & 1.222e-02\\
& rate &  & 1.75 & 1.85 & 1.93 & 1.98\\
& $10^{-8}$ & 2.225e+00 & 6.629e-01 & 1.840e-01 & 4.815e-02 & 1.222e-02\\
& rate &  & 1.75 & 1.85 & 1.93 & 1.98\\
\bottomrule
\end{tabular}
\vspace{1.0em}
\end{table}

\end{example}

\begin{example}\label{example3}
\normalfont
Finally, we consider the problem \eqref{weak-mix-nonhomogeneous} with homogeneous mixed boundary conditions. The limiting linear elasticity problem as $\iota\to0$ is
\begin{equation}\label{homogeneous-linear}
\begin{cases}
-\div\boldsymbol\sigma(\boldsymbol u_0)=\boldsymbol f
& \mbox{in } \Omega,\\[1mm]
\boldsymbol u_0=\boldsymbol 0
& \mbox{on } \Gamma_{\rm cl}\cup\Gamma_{\rm ss},\\[1mm]
\boldsymbol\sigma(\boldsymbol u_0)\boldsymbol n=\boldsymbol 0
& \mbox{on } \Gamma_{\rm fr}.
\end{cases}
\end{equation}
Let the exact solution of the limiting problem \eqref{homogeneous-linear} be the divergence-free function
\[
\boldsymbol{u}_0=\curl\psi
:=
\left(\frac{\partial \psi}{\partial y}, -\frac{\partial \psi}{\partial x}\right)^{\intercal},
\]
where
\[
\psi
=
x^4(1-x)^4(1-4y^3+3y^4)
+
(6x^2-40x^3+90x^4-84x^5+28x^6)y^2(1-y)^2 .
\]
The right-hand side $\boldsymbol{f}$ is computed from \eqref{homogeneous-linear}, and hence is independent of $\lambda$ and $\iota$. We use this $\boldsymbol f$ in \eqref{weak-mix-nonhomogeneous} with homogeneous boundary data. The clamped boundary is $y=1$, the simply supported boundary is $x=0$ and $x=1$, and the free boundary is $y=0$.

The function $\boldsymbol{u}_0$ vanishes on the clamped and simply supported sides, but is generally nonzero on the free side. Moreover, $\boldsymbol{\sigma}(\boldsymbol{u}_0)\boldsymbol n=\boldsymbol 0$ on the free side. Since $\boldsymbol u_0$ does not satisfy the condition $\partial_n\boldsymbol u_0=\boldsymbol 0$ on the clamped side, boundary layers appear there for small $\iota$. On the simply supported and free sides, no additional essential normal-derivative condition is imposed.

In the computation, the displacement condition is imposed strongly, the normal derivative condition on the clamped side is imposed weakly by Nitsche's technique, and the remaining higher-order boundary conditions are treated naturally.

\begin{table}[htbp]
\centering
\caption{Error $\|\boldsymbol{u}_0-\boldsymbol{u}_h\|_{\iota,\lambda,h}$ for Example \ref{example3}.}
\vspace{-1.0em}
\label{table-mix}
\begin{tabular}{ccccccc}
\toprule
$\lambda$ & $\iota\backslash h$ & 1/8 & 1/16 & 1/32 & 1/64 & 1/128\\
\midrule
\multirow{4}{*}{$1$}& $10^{-6}$ & 2.150e-02 & 5.790e-03 & 1.472e-03 & 3.692e-04 & 9.244e-05\\
& rate &  & 1.89 & 1.98 & 2.00 & 2.00\\
& $10^{-8}$ & 2.150e-02 & 5.790e-03 & 1.472e-03 & 3.692e-04 & 9.234e-05\\
& rate &  & 1.89 & 1.98 & 2.00 & 2.00\\
\midrule
\multirow{4}{*}{$10^{6}$}& $10^{-6}$ & 2.954e-02 & 9.141e-03 & 2.525e-03 & 6.544e-04 & 1.654e-04\\
& rate &  & 1.69 & 1.86 & 1.95 & 1.98\\
& $10^{-8}$ & 2.954e-02 & 9.141e-03 & 2.525e-03 & 6.544e-04 & 1.654e-04\\
& rate &  & 1.69 & 1.86 & 1.95 & 1.98\\
\bottomrule
\end{tabular}
\vspace{1.0em}
\end{table}

The numerical results are reported in Table~\ref{table-mix}. For both $\lambda=1$ and $\lambda=10^6$, and for $\iota=10^{-6}$ and $\iota=10^{-8}$, the error $\|\boldsymbol{u}_0-\boldsymbol{u}_h\|_{\iota,\lambda,h}$ converges at the rate $O(h^2)$. Thus the proposed method works well for this mixed boundary case.

Figure~\ref{fig:H2} plots the pointwise Frobenius norm of $\nabla_h\boldsymbol{\varepsilon}_h(\boldsymbol{u}_h)$ for $h=1/128$ and $\lambda=1$. As $\iota$ decreases, a pronounced boundary layer can be observed near the clamped boundary $(y=1)$. In contrast, the simply supported and free boundaries do not exhibit such a pronounced layer structure.

\begin{figure}[htbp]
\centering
\captionsetup[subfigure]{skip=2pt}
\setlength{\tabcolsep}{2pt}

\begin{tabular}{ccc}

\begin{subfigure}{0.32\textwidth}
	\centering
	\includegraphics[width=\linewidth]{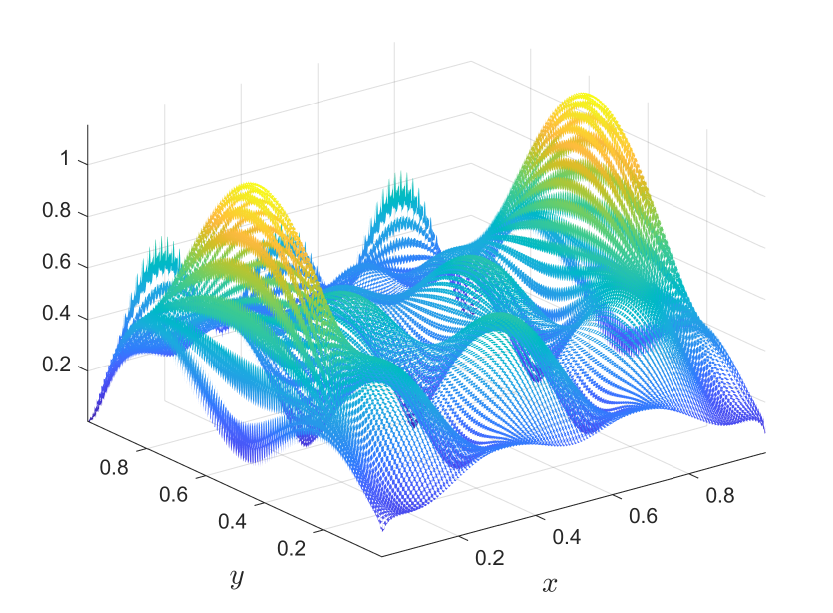}
	\caption{$\iota=10^{-1}$}
\end{subfigure}
\hfill 
\begin{subfigure}{0.32\textwidth}
	\centering
	\includegraphics[width=\linewidth]{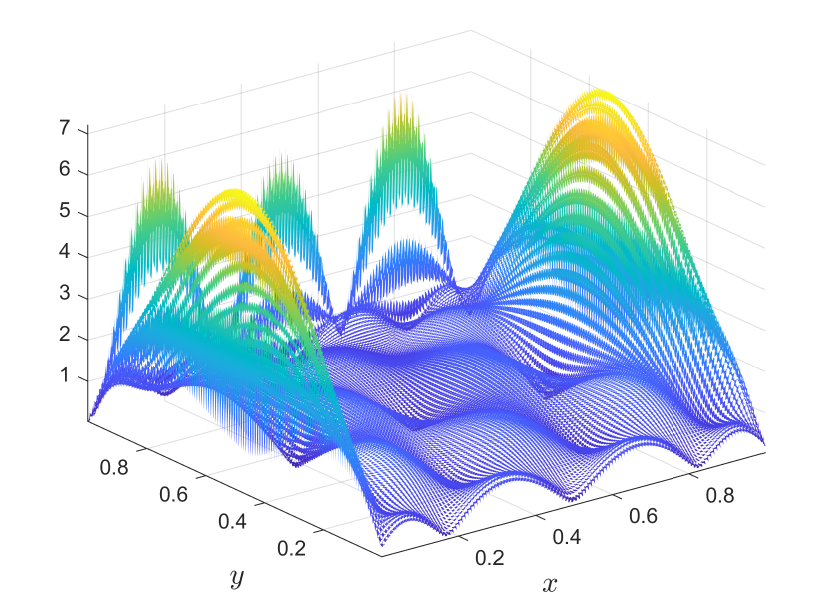}
	\caption{$\iota=10^{-2}$}
\end{subfigure}
\hfill 
\begin{subfigure}{0.32\textwidth}
	\centering
	\includegraphics[width=\linewidth]{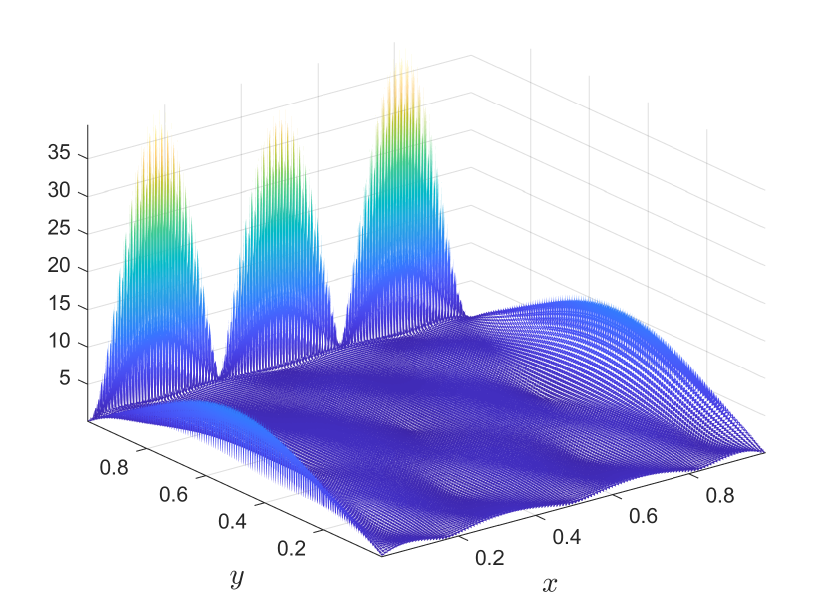}
	\caption{$\iota=10^{-3}$}
\end{subfigure}

\end{tabular}
\caption{The pointwise Frobenius norm of $\nabla_h\boldsymbol\varepsilon_h(\boldsymbol{u}_h)$ for Example~\ref{example3} with $h=1/128$ and $\lambda=1$ for several values of $\iota$.}
\label{fig:H2}
\end{figure}
\end{example}

\appendix
%%%%%%%%%%%%%%%%%%%%%%%%%%%%%%%%%%%%%%%%
%% Finite Element Complexes
%%%%%%%%%%%%%%%%%%%%%%%%%%%%%%%%%%%%%%%%
\section{Nonconforming finite element complexes}\label{app:complexes}
Let $\Omega$ be a contractible bounded polytope. We delve into a nonconforming finite element discretization of the smooth Stokes complex in two dimensions
\begin{equation}
\label{smoothStokescomplex2d}
0\xrightarrow{\subset} W\xrightarrow{\curl}  V
\xrightarrow{\div} Q \xrightarrow {}0
\end{equation}
with $W=H^3(\Omega)\cap H_0^2(\Omega)$, $V=H^2(\Omega;\mathbb{R}^2)\cap H_0^1(\Omega;\mathbb{R}^2)$ and $Q=H^1(\Omega)\cap L_0^2(\Omega)$,
as well as its three-dimensional counterpart,
\begin{equation}
\label{smoothStokescomplex3d}
0\xrightarrow{\subset} H_0^1(\Omega)\xrightarrow{\nabla}  W\xrightarrow{\curl}  V
\xrightarrow{\div} Q \xrightarrow {}0
\end{equation}
with $V=H^2(\Omega;\mathbb{R}^3)\cap H_0^1(\Omega;\mathbb{R}^3)$, $Q=H^1(\Omega)\cap L_0^2(\Omega)$, and
\begin{equation*}
W=\{\boldsymbol{v}\in H_0(\curl,\Omega): \curl\boldsymbol{v}\in V\}.
\end{equation*}
The parts of \eqref{smoothStokescomplex2d}--\eqref{smoothStokescomplex3d} before the last divergence step remain exact in the usual sense. More precisely, in two dimensions,
\begin{equation*}
\curl W=V\cap\ker(\div),
\end{equation*}
and in three dimensions,
\begin{equation*}
\nabla H_0^1(\Omega)=W\cap\ker(\curl),\qquad \curl W=V\cap\ker(\div).
\end{equation*}
The failure of exactness refers to the last step only: in general $\div V\neq Q$; for the two-dimensional polygonal case, see \cite[Theorem~3.1]{ArnoldScottVogelius1988}. In contrast, the finite element complexes constructed below are exact.
%Both complexes~\eqref{smoothStokescomplex2d} and~\eqref{smoothStokescomplex3d} fail to be exact, since in general $\div V\neq Q$; see, e.g., \cite[Theorem~3.1]{ArnoldScottVogelius1988}. In contrast, the finite element complexes constructed below are exact.
The spaces $V$ and $Q$ are discretized by the nonconforming spaces $V_h$ and $Q_h$ introduced in Section \ref{sec3}.
%the previous section.
Therefore, we focus on the construction of the remaining discrete spaces in
\eqref{smoothStokescomplex2d}--\eqref{smoothStokescomplex3d}.

\subsection{Nonconforming finite element complex in two dimensions}
We first construct an $H^3$-nonconforming finite element to discretize $W$ in the complex \eqref{smoothStokescomplex2d}.
Take the space of shape
functions 
\begin{align*}
W(T):=\mathbb{P}_{3}(T)
\oplus\Oplus_{i=0}^{2}(b_{T}b_{F_i}\mathbb{P}_{1}(F_i))
\oplus\Oplus_{i=0}^{2}(b^2_{T}b_{F_i}\mathbb{P}_{0}(F_i)).
\end{align*}
Then $\dim W(T) = 19$.
The DoFs for $W(T)$ are given by
\begin{subequations}\label{2d-wdof}
\begin{align}
\label{2d-wdof1}
v(\texttt{v}), \nabla v(\texttt{v}), & \quad \texttt{v}\in\Delta_0(T),\\
\label{2d-wdof2}
\int_{F}\partial_nv \, q \ds, & \quad q\in\mathbb{P}_1(F), F\in\Delta_{1}(T),\\
\label{2d-wdof3}
\int_{F}\partial_{nn}v \ds, & \quad F\in\Delta_{1}(T),\\
\label{2d-wdof4}
\frac 13\sum_{i=0}^{2}(\Delta v)(\texttt{v}_{i}).&
\end{align}
\end{subequations}
\begin{lemma}
\label{2d-w-dof-uni-solvent}
The DoFs \eqref{2d-wdof} are unisolvent for $W(T)$.
\end{lemma}
\begin{proof}
The number of the DoFs \eqref{2d-wdof} is $9+6+3+1=19=\dim W(T).$

Take $ v\in W(T)$ and assume all the DoFs \eqref{2d-wdof} vanish, then we prove $v=0$. Notice that $v|_F\in\mathbb{P}_{3}(F)$ for each $F\in\Delta_{1}(T)$, we get from the vanishing DoFs (\ref{2d-wdof1}) that $v|_{\partial T} = 0$. 
Then $\curl v\in H_0(\div,T)$. By comparing DoFs \eqref{2d-wdof} with DoFs \eqref{dof}, we conclude $\curl v=0$ from Lemma~\ref{dof-uni-solvent}, $(\curl v)\cdot\boldsymbol{t}=-\partial_nv$, and $\rot(\curl v)=-\Delta v$.
Thus, $v=0$.
\end{proof}

Define the global $H^3$-nonconforming finite element space:
\begin{align*}
W_{h}&=\{v_{h}\in L^{2}(\Omega):v_h|_T\in W(T)~\textrm{for}~T\in\mathcal{T}_h; \textrm{all the DoFs \eqref{2d-wdof1}-\eqref{2d-wdof3}} \\
&\qquad\quad\textrm{are~single-valued, and DoFs \eqref{2d-wdof1}-\eqref{2d-wdof2} on boundary vanish}\}.
\end{align*}
We have $W_h\subset H_0^1(\Omega)$ and $W_h\not\subseteq H_0^2(\Omega)$.

\begin{lemma}
The nonconforming discrete complex
\begin{equation}
\label{2dcomplex-2}
0\xrightarrow{\subset} W_h\xrightarrow{\curl}  V_h
\xrightarrow{\div} Q_h \xrightarrow {}0
\end{equation}
is exact.
\end{lemma}
\begin{proof}
Clearly, the sequence \eqref{2dcomplex-2} is a complex.
The construction of the local spaces and Remark~\ref{rem:bubble-curl} imply
\[
\curl W(T)=V(T)\cap\ker(\div).
\]

For $\boldsymbol{v}_h\in V_h\cap\ker(\div)\subset H_0(\div,\Omega)\cap\ker(\div)$,
it follows that $\boldsymbol{v}_h=\curl w_h$ with $w_h\in H_0^1(\Omega)$ satisfying $w_h|_T\in W(T)$ for $T\in\mathcal T_h$. The vertex gradients of $w_h$ are single-valued by DoF \eqref{dof1} of $V_h$. The continuity of its normal-derivative moments and second normal-derivative means follows from DoFs \eqref{dof3} and \eqref{dof5}, respectively. On the boundary, the vertex values, vertex gradients, and moments of the first normal derivative all vanish. Hence $w_h\in W_h$. This concludes $V_h\cap\ker(\div)=\curl W_h$.
Then we end the proof by using \eqref{divontodiscrete}.
\end{proof}

Define interpolation operator $I_h^W: W\rightarrow W_h$ as follows:
\begin{align*}
(I_h^Wv)(\texttt{v})&=v(\texttt{v}) \qquad\qquad\qquad\qquad\qquad\;\forall~\texttt{v}\in \mathring{\mathcal{V}}_h, \\
\curl(I_h^Wv)(\texttt{v})&=I_h^V(\curl v)(\texttt{v}) \qquad\qquad\quad\quad\;\,\forall~\texttt{v}\in \mathring{\mathcal{V}}_h, \\
\int_{F}\curl(I_h^Wv) \cdot \boldsymbol{t}\,q\ds&=\int_{F}I_h^V(\curl v) \cdot \boldsymbol{t}\,q\ds \qquad\quad\;\, \forall~q\in\mathbb{P}_1(F), F\in\mathring{\mathcal{F}}_h,\\
\int_{F}\partial_n(\curl(I_h^Wv) \cdot \boldsymbol{t})\ds&=\int_{F}\partial_n(I_h^V(\curl v) \cdot \boldsymbol{t})\ds \qquad\; \forall~F\in\mathcal{F}_h,\\
\sum_{i=0}^{2}(\Delta(I_h^Wv))(\texttt{v}_{i})&=-\sum_{i=0}^{2}(\rot(I_h^V(\curl v)))(\texttt{v}_{i}) \quad \forall~T\in\mathcal{T}_h.
\end{align*}
It can be verified that
\begin{equation}\label{commutative2dcurl}
\curl(I_h^Wv)=I_h^V(\curl v)\quad\forall~ v\in W.
\end{equation}

The combination of complex \eqref{2dcomplex-2}, \eqref{commutative2dcurl} and \eqref{commutative1} yields the commutative diagram
\[
\xymatrix{
0 \ar[r]^{\subset\hspace{1em}} &
W \ar[r]^{\curl\hspace{1em}} \ar[d]^{I^W_h} &
V \ar[r]^{\hspace{0.5em}\div} \ar[d]^{I^{V}_h} &
Q \ar[r] \ar[d]^{I^{Q}_h} & 0\\
0 \ar[r]^{\subset} & W_h \ar[r]^{\curl} &
V_h \ar[r]^{\div} & Q_h \ar[r] & 0. 
}
\]

\subsection{Nonconforming finite element complex in three dimensions}
Now we consider the finite element discretization of the complex \eqref{smoothStokescomplex3d}.

To discretize space $W$,
take the space of shape
functions 
\begin{align*}
W(T):=\mathbb{P}_{3}^-(T;\mathbb{R}^{3})
&\oplus\Oplus_{i=0}^{3}(b_{T}b_{F_i}\mathbb{P}_{1}(F_i)\otimes\mathscr T^{F_i})\oplus\Oplus_{i=0}^{3}(b^2_{T}b_{F_i}\mathbb{P}_{0}(F_i)\otimes\mathscr T^{F_i})\\
&\oplus\Oplus_{e\in\Delta_1(T)}\big(b_e^2\mathbb P_0(e)(\boldsymbol m_e\times\boldsymbol t_e)\big),
\end{align*}
where $\mathbb{P}_{3}^-(T;\mathbb{R}^{3})=\nabla\mathbb{P}_{3}(T)\oplus(\mathbb{P}_{2}(T;\mathbb R^3)\times\boldsymbol{x})$, and $\mathscr T^{F_i}$ is the tangent plane of $F_i$.
By Remark~\ref{rem:bubble-curl}, the curl maps the bubble summands onto the corresponding summands of $V(T)$. Their directness and the coefficient-elimination argument in Lemma~\ref{dof-uni-solvent} imply injectivity on each bubble summand. Polynomial de Rham exactness then gives $\ker(\curl|_{W(T)})=\nabla\mathbb P_3(T)$ and the directness of the defining sum of $W(T)$.
We have $\dim W(T) = 83$.
The DoFs for $W(T)$ are given by
\begin{subequations}\label{3d-wdof}
\begin{align}
\label{3d-wdof1}
\int_{e}\boldsymbol{v} \cdot\boldsymbol{t}\,q \ds, & \quad q\in\mathbb{P}_2(e), e\in\Delta_{1}(T),\\
\label{3d-wdof2}
\int_{e}\curl\boldsymbol{v} \ds, & \quad e\in\Delta_{1}(T),\\
\label{3d-wdof3}
\int_{F}(\Pi_F\boldsymbol{v}) \cdot \boldsymbol{q} \ds, & \quad \boldsymbol{q}\in\boldsymbol{x}\mathbb{P}_0(F), F\in\Delta_{2}(T),\\
\label{3d-wdof4}
\int_{F}\curl\boldsymbol{v} \cdot\boldsymbol{n} \, q \ds, & \quad q\in\hat{\mathbb{P}}_2(F), F\in\Delta_{2}(T),\\
\label{3d-wdof5}
\int_{F}(\curl\boldsymbol{v}) \cdot \boldsymbol{t}_i \, q \ds, & \quad q\in\mathbb{P}_1(F), F\in\Delta_{2}(T), i=1,2,\\
\label{3d-wdof6}
\int_{F}\partial_{n}((\curl\boldsymbol{v})\cdot \boldsymbol{t}_i) \ds, & \quad F\in\Delta_{2}(T), i = 1,2,\\
% \label{3d-wdof7}
% \int_{T}\boldsymbol{v}\cdot \boldsymbol{q} \ds, & \quad \boldsymbol{q}\in\boldsymbol{x}\mathbb{P}_0(T),\\
\label{3d-wdof7}
\frac{1}{4}\sum_{i=0}^{3}(\curl^2\boldsymbol{v})(\texttt{v}_{i}).&
\end{align}
\end{subequations}

\begin{lemma}
\label{3d-w-dof-uni-solvent}
The DoFs \eqref{3d-wdof} are unisolvent for $W(T)$.
\end{lemma}
\begin{proof}
The number of the DoFs \eqref{3d-wdof} is
\[
6\times(3+3)+4\times(1+2+6+2)+3=83=\dim W(T).
\]

Take $ \boldsymbol{v}\in W(T)$ and assume all the DoFs \eqref{3d-wdof} vanish,
then we prove $\boldsymbol{v}=0$. 
Applying integration by parts, we get from the vanishing DoF \eqref{3d-wdof1} that
\begin{equation*}
\int_{F}\curl\boldsymbol{v} \cdot\boldsymbol{n}\ds=0  \quad \forall~F\in\Delta_{2}(T).
\end{equation*}
Using the vanishing DoFs \eqref{3d-wdof2} and \eqref{3d-wdof4}-\eqref{3d-wdof7}, we apply Lemma~\ref{dof-uni-solvent} to acquire $\curl\boldsymbol{v}=0$.
So $\boldsymbol{v}=\nabla w$ with $w\in\mathbb P_3(T)$ satisfying $w(\texttt{v}_0)=0$.
DoF \eqref{3d-wdof1} implies $w=0$ on all edges, so $w|_F=c_Fb_F$. By \eqref{3d-wdof3} and face integration by parts, $0=\int_F\nabla_Fw\cdot\boldsymbol x\ds=-2\int_Fw\ds$, hence $c_F=0$. A cubic polynomial vanishing on all four faces is zero, so $w=0$ and $\boldsymbol v=0$.
\end{proof}

Define the global $W$-nonconforming finite element space
\begin{align*}
W_{h}&=\{\boldsymbol{v}_{h}\in L^{2}(\Omega;\mathbb{R}^{3}):\boldsymbol{v}_h|_T\in W(T)~\textrm{for}~T\in\mathcal{T}_h; \textrm{all~the~DoFs \eqref{3d-wdof1}-\eqref{3d-wdof6}} \\
&\qquad\qquad\;\;\textrm{are~single-valued, and DoFs \eqref{3d-wdof1}-\eqref{3d-wdof5} on boundary vanish}\}. %,\\
% V_{h0}&=\{\boldsymbol{v}_{h}\in  V_{h}: \textrm{DoFs \eqref{dof4}-\eqref{dof5} on boundary vanish}\}.
\end{align*}
We have $W_{h}\subset H_0(\curl,\Omega)$, but $W_{h}\not\subseteq W$.

\begin{lemma}
The nonconforming discrete complex
\begin{align}
\label{3dcomplex-2}
0\xrightarrow{\subset}V_h^L\xrightarrow{\nabla} W_h \xrightarrow{\curl}  V_h
\xrightarrow{\div} Q_h \xrightarrow {}0
\end{align}
is exact, where $V^L_h:=\{v_h\in H_0^1(\Omega): v_h|_T\in\mathbb{P}_3(T)~\textrm{for~each}~T\in\mathcal{T}_h\}.$
\end{lemma}
\begin{proof}
First, the sequence \eqref{3dcomplex-2} is a complex.
By \eqref{divontodiscrete}, $\div V_h=Q_h$, which means
\begin{equation*}
\dim V_h\cap\ker(\div)=\dim V_h -\dim Q_h=3\#\mathring{\mathcal{E}}_h +9\#\mathring{\mathcal{F}}_h+2\#\mathcal{F}_h+2\#\mathcal{T}_h+1.
\end{equation*}

Next we verify $W_h\cap\ker(\curl)=\nabla V^L_h$.
For $\boldsymbol{w}_h\in W_h\cap\ker(\curl)\subset H_0(\curl,\Omega)\cap\ker(\curl)$,
it follows that $\boldsymbol{w}_h=\nabla v_h$ with $v_h\in H_0^1(\Omega)$ satisfying $v_h|_T\in \mathbb{P}_3(T)$ for $T\in\mathcal T_h$. Then $v_h\in V^L_h$, and $\boldsymbol{w}_h\in\nabla V^L_h$.
As a result,
\begin{equation*}
\dim\curl W_h=\dim W_h-\dim V^L_h= 4\#\mathring{\mathcal{E}}_h+8\#\mathring{\mathcal{F}}_h+2\#\mathcal{F}_h+3\#\mathcal{T}_h-\#\mathring{\mathcal{V}}_h.
\end{equation*} 

Now we have
\begin{equation*}
\dim V_h\cap\ker(\div) - \dim\curl W_h= \#\mathring{\mathcal{V}}_h-\#\mathring{\mathcal{E}}_h +\#\mathring{\mathcal{F}}_h-\#\mathcal{T}_h+1,
\end{equation*}
by the Euler's formula, which equals $0$. Therefore, $V_h\cap\ker(\div)=\curl W_h$.
\end{proof}

Let $I_h^L: H_0^1(\Omega)\to V_h^L$ be the Scott-Zhang interpolation operator \cite{ScottZhang1990}.
We follow the idea in \cite[Section 4.2]{MR4621133} to construct an interpolation operator $I_h^W: W\to W_h$. First define an interpolation operator $\widetilde{I}_h^W: W\cap H_0^2(\Omega;\mathbb R^3)\to W_h$ as follows: 
\begin{align*}
\int_{e}(\widetilde{I}_h^W\boldsymbol{v})\cdot\boldsymbol{t}\,q \ds &= \int_{e}\boldsymbol{v} \cdot\boldsymbol{t}\,q \ds, \qquad\qquad\quad\;\;\;\forall~q\in\mathbb{P}_2(e), e\in\mathring{\mathcal{E}}_h,\\
\int_{e}\curl(\widetilde{I}_h^W\boldsymbol{v})\ds &= \int_{e}I_h^V(\curl\boldsymbol{v})\ds, \qquad\quad\;\;\,\forall~e\in\mathring{\mathcal{E}}_h, \\
\int_{F}(\Pi_F(\widetilde{I}_h^W\boldsymbol{v})) \cdot \boldsymbol{q} \ds &= \int_{F}(\Pi_F\boldsymbol{v}) \cdot \boldsymbol{q} \ds, \qquad\qquad\forall~ \boldsymbol{q}\in\boldsymbol{x}\mathbb{P}_0(F), F\in\mathring{\mathcal{F}}_h,\\
\int_{F}\curl(\widetilde{I}_h^W\boldsymbol{v}) \cdot\boldsymbol{n} \, q \ds &= \int_{F}I_h^V(\curl\boldsymbol{v}) \cdot\boldsymbol{n} \, q \ds,  \quad \;\;\forall~q\in\hat{\mathbb{P}}_2(F), F\in\mathring{\mathcal{F}}_h,\\
\int_{F}\Pi_F\curl(\widetilde{I}_h^W\boldsymbol{v}) \cdot \boldsymbol{q} \ds &= \int_{F}\Pi_FI_h^V(\curl\boldsymbol{v}) \cdot \boldsymbol{q} \ds,  \;\;\;\forall~ \boldsymbol{q}\in\mathbb{P}_1(F;\mathscr T^F), F\in\mathring{\mathcal{F}}_h,\\
\int_{F}\partial_{n}(\Pi_F\curl(\widetilde{I}_h^W\boldsymbol{v})) \ds &= \int_{F}\partial_{n}(\Pi_FI_h^V(\curl\boldsymbol{v})) \ds,  \;\forall~F\in\mathcal{F}_h, \\
\sum_{i=0}^{3}(\curl^2(\widetilde{I}_h^W\boldsymbol{v}))(\texttt{v}_{i}) &=\sum_{i=0}^{3}(\curl I_h^V(\curl\boldsymbol{v}))(\texttt{v}_{i}), \;\;\forall~T\in\mathcal{T}_h.
\end{align*}
It can be verified that
\begin{equation}\label{commutative3dcurl0}
\curl(\widetilde{I}_h^W\boldsymbol{v})=I_h^V(\curl\boldsymbol{v})\quad\forall~\boldsymbol{v}\in W\cap H_0^2(\Omega;\mathbb R^3).
\end{equation}

Applying Lemma 4.5 in \cite{MR4621133}, we have the regular decomposition
\begin{equation}\label{Wregulardecomp}
W = (W\cap H_0^2(\Omega;\mathbb R^3)) + \nabla H_0^1(\Omega).
\end{equation}
For $\boldsymbol{v}\in W$, by the regularity decomposition \eqref{Wregulardecomp}, write $\boldsymbol{v}=\boldsymbol{v}_2+\nabla v_1$ with $\boldsymbol{v}_2\in W\cap H_0^2(\Omega;\mathbb R^3)$ and $v_1\in H_0^1(\Omega)$. Here $\boldsymbol{v}_2$ satisfies equation (4.19) in \cite{MR4621133}.
Since $\curl\boldsymbol v_2=\curl\boldsymbol v\in V$, the regular component $\boldsymbol v_2$ indeed belongs to $W\cap H_0^2(\Omega;\mathbb R^3)$.
When $\boldsymbol{v}\in\nabla H_0^1(\Omega)$, $\boldsymbol{v}_2=0$.
Define the interpolation 
\[
I_h^W\boldsymbol{v}=\widetilde{I}_h^W\boldsymbol{v}_2+\nabla(I_h^Lv_1).
\]
By the definition of $I_h^W$,
\begin{equation}\label{commutative3dgrad}
I_h^W(\nabla v) = \nabla(I_h^Lv)\quad\forall~v\in H_0^1(\Omega).
\end{equation}
It follows from \eqref{commutative3dcurl0} that
\begin{equation}\label{commutative3dcurl}
\curl(I_h^W\boldsymbol{v})=\curl(\widetilde{I}_h^W\boldsymbol{v}_2)=I_h^V(\curl\boldsymbol{v})\quad\forall~\boldsymbol{v}\in W.
\end{equation}
The combination of complex \eqref{3dcomplex-2}, \eqref{commutative1} and \eqref{commutative3dgrad}-\eqref{commutative3dcurl} gives the commutative diagram
\[
\xymatrix{
0 \ar[r]^{\subset} &
H_0^1(\Omega) \ar[r]^{\nabla} \ar[d]^{I^L_h} &
W \ar[r]^{\curl} \ar[d]^{I^{W}_h} &
V \ar[r]^{\div} \ar[d]^{I^{V}_h} &
Q \ar[r] \ar[d]^{I^{Q}_h} & 0\\
0 \ar[r]^{\subset} & V^L_h \ar[r]^{\nabla} &
W_h \ar[r]^{\curl} &  V_h \ar[r]^{\div} & Q_h \ar[r] & 0. 
}
\]

\bibliographystyle{abbrv} % F6+F8+F6+F8
\bibliography{ref} 
\end{document}